\documentclass[12pt]{article}
\usepackage[a4paper, margin=1.1in]{geometry}
\usepackage{amsmath,amsthm}
\usepackage{amsfonts,amssymb}
\usepackage{graphicx}
\usepackage{float}
\newtheorem{op}{Operation}[section]
\usepackage{caption}
\usepackage{chngcntr}
\usepackage{fncylab}
\usepackage{verbatim}
\usepackage{cite}
\newtheorem{thm}{Theorem}[section]
\newtheorem{cor}[thm]{Corollary}

\newtheorem{prop}{Proposition}[section]

\newtheorem{defn}{Definition}[section]
\theoremstyle{remark}
\newtheorem{rem}{Remark}[section]

\newtheorem{exam}{Example}[section]

\begin{document}

\title
{\bf{Construction of equienergetic and Randi{\'c} equienergetic graphs}}
\author {\small Jahfar T K \footnote{jahfartk@gmail.com} and Chithra A V \footnote{chithra@nitc.ac.in} \\ \small Department of Mathematics, National Institute of Technology Calicut, Kerala, India-673601}
\date{ }
\maketitle
\begin{abstract}
	In this paper, we give several constructions for the pairs of graphs to be equienergetic and Randi{\'c} equienergetic graphs. Also, some new families of integral and Randi{\'c} integral graphs are obtained. As an application, a sequence of graphs established with reciprocal eigenvalue property and anti-reciprocal eigenvalue property.
\end{abstract}

\hspace{-0.6cm}\textbf{AMS classification}: 05C50, 05C76
\section{Introduction}
In this paper, we consider simple connected graphs. Let $G=(V,E)$ be a simple graph of order $p$ and size $q$ with vertex set $V(G)=\{v_{1}, v_{2},...,v_{p}\}$ and edge set $E(G)=\{e_{1}, e_{2},...,e_{q}\}$. The degree of a vertex $v_i$ in $G$ is the number of edges incident to it and is denoted by $d_i=d_G(v_i)$. A graph $G$ is called a regular graph, if all the vertices have the same degree. The path, complete graph and star graph on $p$ vertices are denoted by $P_p$,  $K_p$ and $K_{1,p-1}$ respectively. 
 The adjacency matrix $A(G)=[a_{ij}]$ of $G$ is a square symmetric matrix of order $p$ whose $(i,j)^{th}$ entry is equal to one if the vertices $v_{i}$ and $v_{j}$  are adjacent, and is equal to zero otherwise. The   eigenvalues of a graph $G$ are defined as the eigenvalues of its adjacency matrix $A(G)$. Denote the eigenvalues of $A(G)$ by $ \lambda_{1},\lambda_{2},...,\lambda_{p}$.   Let $ \lambda_{1},\lambda_{2},...,\lambda_{t}$ be the distinct eigenvalues of $G$ with multiplicities $m_1,m_2,...,m_t$  respectively, then the spectrum of $G$ is denoted by \[spec(G)=\begin{pmatrix}
\lambda_1&\lambda_2& ...&\lambda_t\\

m_1&m_2&...&m_t\\
\end{pmatrix}.\]
The energy of a graph $G$ was first introduced by Gutman \cite{gutman1978energy}  in 1978 and is defined as the sum of the absolute values of the eigenvalues of its adjacency matrix , $\varepsilon(G)=\displaystyle\sum_{i=1}^{p} |{\lambda_i}|.$ The concept of graph energy arose in chemistry, (see \cite{cvetkovic1980spectra,gutman1978energy}). The Randi{\'c} matrix $R(G)$ of the graph $G$ is a square matrix of order $p$ whose  $(i,j)^{th}$ entry is equal to $\frac{1}{\sqrt{d_id_j}}$ if the vertices $v_{i}$ and $v_{j}$  are adjacent, and is equal to zero otherwise. The  eigenvalues of $R(G)$ are called Randi{\'c} eigenvalues of $G$ and it is denoted by
$\rho_{i},1\leq i \leq p$.  Let $ \rho_{1},\rho_{2},...,\rho_{s}$ be the distinct Randi{\'c} eigenvalues of $G$ with multiplicities $m_1,m_2,...,m_s$ respectively,  then the Randi{\'c} spectrum of $G$ is denoted by \[RS(G))=\begin{pmatrix}
\rho_1&\rho_2& ...&\rho_s\\

m_1&m_2&...&m_s\\
\end{pmatrix}.\] If $G$ has no isolated vertices, then $R(G)=D^{-1/2}A(G)D^{-1/2}$ where $D$  is the diagonal matrix of vertex degrees of $G$\cite{bozkurt2010randic}. Randi{\'c} energy of $G$ is defined as $\varepsilon_R(G)= \displaystyle\sum_{i=1}^{p} |{\rho_i}|$\cite{bozkurt2010randic,gutman2014randic,cvetkovic2009introduction}. The incidence matrix of a graph $G$, $I(G)$ is the $p\times q$ matrix whose $(i,j)^{th}$ entry is 1 if $v_i$ is incident to $e_j$ and 0 otherwise. The rank of the incidence matrix $I(G)$ is $p-1$ if $G$ is bipartite and $p$ otherwise\cite{bapat2010graphs}. A graph $G$ is said to be integral if the eigenvalues of its adjacency matrix are all integers. Two non-isomorphic graphs are said to be cospetral if they have the same spectra, otherwise, they are known as non cospectral. 
Two non-isomorphic graphs $G_1$ and $G_2$ of the same order are said to be equienergetic if $\varepsilon (G_1)=\varepsilon(G_2)$ \cite{ramane2007construction}. In analogous to equienergetic graphs, two non-isomorphic graphs of same order are said to be Randi{\'c} equienergetic if they have the same Randi{\'c} energy.  
\par A graph $G$ is said to be singular (respectively, nonsingular) if $A(G)$ is singular (respectively, nonsingular). A nonsingular graph $G$ is said to have the reciprocal eigenvalue property (R) if for each eigenvalue $\lambda$ of adjacency matrix A(G), its reciprocal $\frac{1}{\lambda}$ is also an eigenvalue of $A(G)$\cite{barik2006nonsingular}. Moreover if $\lambda$ and $\frac{1}{\lambda}$ has the same multiplicity, then that property is known as strong reciprocal eigenvalue property (SR). A nonsingular graph $G$ is said to have the anti-reciprocal eigenvalue property (-R) if for each eigenvalue $\lambda$ of adjacency matrix A(G), its negative reciprocal $-\frac{1}{\lambda}$ is also an eigenvalue of $A(G)$\cite{ahmad2020class}. In addition,  if $\lambda$ and $-\frac{1}{\lambda}$ have the same multiplicity, then that property is known as strong anti-reciprocal eigenvalue property (-SR)\cite{ahmad2019noncorona}.
 \par The rest of the paper is organized as follows. In Section \ref{s2}, we give a list of some previously known results. In  Section 3, we determine the energy and Randi{\'c} energy of graphs obtained from a graph by other unary operations. Also, new classes of integral graphs are obtained. In Section 4, we construct  new families  of equienergetic and Randi{\'c} equienergetic non-cospectral graphs. In Section 5, we establish some family of graphs with reciprocal eigenvalue property and anti-reciprocal eigenvalue property. 
 \section{Preliminaries}\label{s2}
 In this section, we start with some definitions and terminology required for the discussions in subsequent sections.
 \begin{defn}\textnormal{\cite{cvetkovic1980spectra}}
 	The Kronecker product of two graphs  $G_1$ and $G_2$ is a graph $ G_1\times G_2$ with vertex set $V(G_1) \times V(G_2)$ and the vertices $(x_1,x_2)$ and $(y_1,y_2)$ are adjacent if and only if $(x_1,y_1)$ and $(x_2,y_2)$ are edges in  $G_1$ and $G_2$ respectively.
 \end{defn}
 \begin{defn}\textnormal{\cite{cvetkovic1980spectra}}
 	Let $A\in R^{ m \times n} $,$B\in R^{ p \times q}. $ Then the Kronecker product of $A$ and $B$ is defined as follows \[ A\otimes B=\begin{bmatrix}
 	a_{11}B & a_{12}B & \dots  & a_{1n}B \\
 	a_{21}B & a_{22}B &  \dots  & a_{2n}B \\
 	\vdots & \vdots &  \ddots & \vdots \\
 	a_{m1}B & a_{m2}B &  \dots  & a_{mn}B\\
 	\end{bmatrix}.\]
 	
 \end{defn}
 \begin{prop}\textnormal{\cite{cvetkovic1980spectra}}
 	Let $A,B\in R^{n\times n}$. Let $\lambda$ be an eigenvalue of matrix $A$ with corresponding eigenvector $x$ and	$\mu$ be an eigenvalue of matrix $B$ with corresponding eigenvector $y$, then $\lambda\mu$ is an eigenvalue of $A\otimes B$ with corresponding eigenvector $x\otimes y.$
 \end{prop}
\begin{defn}\textnormal{\cite{cvetkovic1980spectra}}
	The line graph, $L(G)$ of a graph $G$ has $E(G)$ as its vertex set and two vertices are adjacent in $L(G)$ if and only if the corresponding edges in $G$ are incident to a common vertex.
\end{defn}
\begin{thm}\textnormal{\cite{ramane2004equienergetic}}
	Let $G_1$ and $G_2$ be two non-cospectral $r_1$ regular graphs, $r_1\geq 3$ with $p$ vertices.   Then for any $k\geq 2$, $L^k(G_1)$ and  $L^k(G_2)$ are regular non-cospectral and equienergetic graphs.
\end{thm}
\begin{defn}\textnormal{\cite{vaidya2017energy}}
	The $m$-shadow graph $D_m(G)$ of a connected graph $G$ is constructed by taking m copies of $G$  say, $G_1,G_2,...,G_m$, then join each vertex $u$ in $G_i$ to the neighbors of the corresponding vertex $v$  in  $G_j,1\leq i, j\leq m.$\\
	The adjacency matrix of $m$-shadow graph of $G$ is  
	\[ A(D_m(G)) =\begin{bmatrix}
	A(G) & A(G) &   \dots  & A(G)\\
	A(G) & A(G) &  \dots  & A(G) \\
	\vdots & \vdots &  \ddots & \vdots \\
	A(G) & A(G) &  \dots  & A(G)\\
	\end{bmatrix}_{mp}.\]\\
	\begin{prop}\textnormal{\cite{vaidya2017energy}}
		The energy of $m$-shadow graph of $G$ is, $\varepsilon(D_m(G))=m\varepsilon(G).$
	\end{prop}
\end{defn}
\begin{defn}\textnormal{\cite{cvetkovic1980spectra}}
		The subdivision graph of a graph $G$ is obtained by inserting new vertices between every edges of graph $G$. It is denoted by $S(G).$	
\end{defn} Let $G$
be a simple $(p,q)$  graph. Then the number of vertices and edges in $S(G)$ are $p+q$ and $2q$ respectively. The adjacency matrix of  $S(G)$ is   
\[A(S(G))=\begin{bmatrix}
O_{p\times p}&I(G)\\
(I(G))^T&O_{q\times q}\\
\end{bmatrix}. 
\]
where O is null matrix and $I(G)$ is the incidence matrix of $G.$ Let $r$ be the rank of $I(G)$, then the rank of $A(S(G))$ is $2r$. The non-zero eigenvalues of $S(G)$ are denoted by $\tau_1,\tau_2,...,\tau_{2r}$. If $G$ is $r_1-$regular, then the eigenvalues of $S(G)$ are $\pm\sqrt{\lambda_i+r_1},i=1,2,...,p$ and $0$ with multiplicity $q-p$\cite{cvetkovic1980spectra}. The Randi{\'c} matrix of $S(G)$ is \begin{align*}
R(S(G))=&\begin{bmatrix}
D^{-\frac{1}{2}}&O\\
O&(2I_q)^{-\frac{1}{2}}\\
\end{bmatrix}\begin{bmatrix}
O & I(G)\\
(I(G))^T &O\\
\end{bmatrix}\begin{bmatrix}
D^{-\frac{1}{2}}&O\\
O&(2I_q)^{-\frac{1}{2}}\\
\end{bmatrix}\\
=&\begin{bmatrix}
O&D^{-\frac{1}{2}}I(G)(2I_q)^{-\frac{1}{2}}\\
(2I_q)^{-\frac{1}{2}}(I(G))^TD^{-\frac{1}{2}}&O\\
\end{bmatrix}.
\end{align*}
Note that the rank of $R(S(G))$ is $2r,$ its non-zero Randi{\'c} eigenvalues are denoted by $\gamma_1,\gamma_2,...,\gamma_{2r}.$

\begin{prop}\textnormal{\cite{gu2014general}}
	If the graph $G$ is $r_1-$regular, then $\varepsilon_R(G)= \frac{\varepsilon(G)}{r_1}.$
\end{prop}
\begin{defn}\textnormal{\cite{hou2010spectrum}}
	Let $G_1$ and $G_2$ be two graphs on disjoint sets of $p_1$ and $p_2$ vertices, $q_1$ and $q_2$ edges, respectively. The edge corona $G_1\diamond G_2$ of $G_1$ and $G_2$ is defined as the graph obtained by taking one copy of $G_1$ and $q_1$ copies of $G_2$, and then joining two end vertices of the $i^{th}$edge of $G_1$ to every vertex in the $i^{th}$ copy of $G_2$.
\end{defn}
\section{Randi{\'c} energy of specific graphs}
\par In this section, some new  operations on $G$ are defined, the spectrum and energy of the resultant graphs  are determined. Also, we compute the Randi{\'c} spectrum and Randi{\'c} energy of these new graphs. Moreover, equienergetic and Randi{\'c} equienergetic  graphs are constructed. In addition, we provide some new families of integral graphs.\\\par The following operation is obtained by taking $G_2=\overline{K}_{p}$ in $G_1\diamond G_2$ and removing the edges of $G_1$.

\begin{op}\label{op1}
	\par Let $G$ be a simple $(p,q)$ graph with vertex set $V(G)=\{v_1,v_2,...,v_p\}$ and edge set $E(G)=\{e_1,e_2,...,e_q\}.$
	Corresponding to every edge $e_i$,$1 \leq i \leq q$ in $G$,  introduce a set $U^k_i$ of $k$ (positive integer) isolated vertices and make every vertex in $U^k_i$ adjacent to the vertices incident with $e_i, i=1,2,...q$ and remove edges of $G$  only. The resultant graph is denoted by $S(G)_k.$ \\\par The number of vertices and edges of the graph $S(G)_k$ are $p+kq$
	and $2kq$ respectively.  \\If $k=1$, then $S(G)_k$ coincides with the subdivision  graph $S(G)$. \\ The following figure illustrate the above operation.

\begin{figure}[H]
	\centering
	\includegraphics[width=5.0cm]{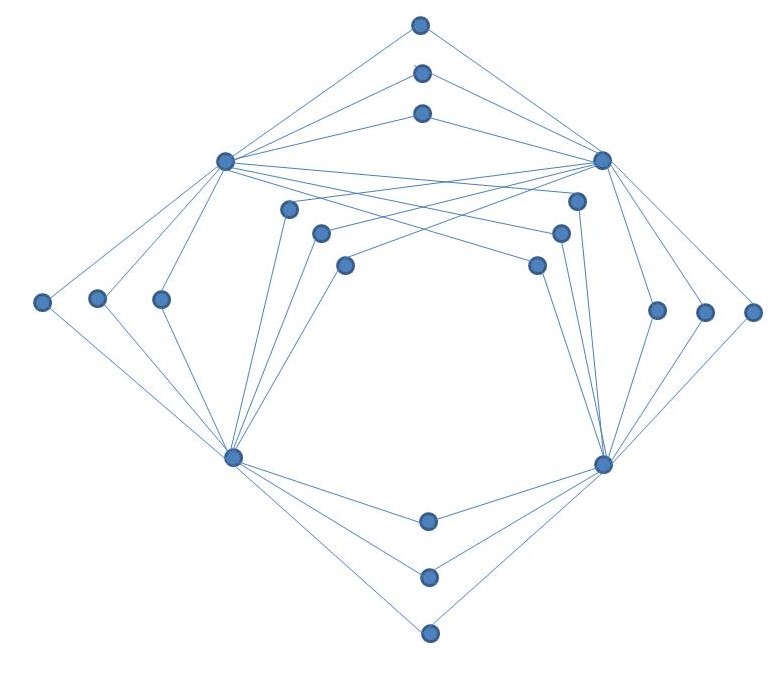}
	\caption{$S(K_4)_3$}
	\label{pict26.jpg}
\end{figure} 
\end{op}
Let $G$ be a simple $(p,q)$  graph. Using the suitable labeling of the vertices of $S(G)_k$, the adjacency matrix of $S(G)_k$  
 is  	
 \[ A(S(G)_k)=\begin{bmatrix}
	O & I(G) &  I(G) & \dots  & I(G)\\
	(I(G))^T & O & O & \dots  & O \\
	(I(G))^T & O & O & \dots  & O \\
	\vdots & \vdots & \vdots & \ddots & \vdots \\
	(I(G))^T & O & O & \dots  & O\\
	\end{bmatrix}_{p+kq}.\]	
	The degree matrix of $S(G)_k$ is\[D(S(G)_k)=\begin{bmatrix}
		kD & O &  O & \dots  & O\\
		O & 2I_q & O & \dots  & O \\
		O & 2I_q & O & \dots  & O \\
		\vdots & \vdots & \vdots & \ddots & \vdots \\
		O & O & O & \dots  & 2I_q\\
	\end{bmatrix}.\]\\
	The Randi{\'c} matrix of $S(G)_k$ is
	\[R(S(G)_k)=\begin{bmatrix}
		O& (kD)^{-\frac{1}{2}}I(G)(2I_q)^{-\frac{1}{2}} &  \dots  & (kD)^{-\frac{1}{2}}I(G)(2I_q)^{-\frac{1}{2}}\\
		(2I_q)^{-\frac{1}{2}}(I(G))^T(kD)^{-\frac{1}{2}} & O &  \dots  & O \\
		\vdots & \vdots &  \ddots & \vdots \\
		(2I_q)^{-\frac{1}{2}}(I(G))^T(kD)^{-\frac{1}{2}} & O & \dots  & O\\
	\end{bmatrix}.\]
The following theorem gives a relation between the eigenvalues of $S(G)$ and $S(G)_k$.

\begin{thm}\label{th1}
	Let $G$  be a simple $(p,q)$ graph and $S(G)$ be the subdivision graph of $G$. Then the spectrum of $S(G)_k$ is	\[spec(S(G)_k)=\begin{pmatrix}
	 0 & \sqrt{k}\tau_1 & \sqrt{k}\tau_2 & \dots &\sqrt{k}\tau_{2r}\\
	 p+kq-2r &1 & 1 &\dots & 1\\
	\end{pmatrix}\] where $\tau_i$'s, $1\leq i\leq 2r,$  are the non-zero eigenvalues of  $S(G)$.
\end{thm}

\begin{proof}
	Let $X$ and $Y$ be column matrix with orders $p\times1$ and $q\times 1$ respectively and let $U=\begin{bmatrix}
		X\\
		Y\\
	\end{bmatrix}_{(p+q)\times1}$ be the eigenvector corresponding to the non-zero eigenvalue $\tau_i$, $1\leq i\leq 2r,$ of $S(G).$ Then $A(S(G))U=\tau_i U.$\\ That is,
	$\begin{bmatrix}
		O&I(G)\\
		(I(G))^T&O\\
	\end{bmatrix}_{(p+q)\times(p+q)}\begin{bmatrix}
	X\\
	Y\\
	\end{bmatrix}_{(p+q)\times1}=\tau_i\begin{bmatrix}
	X\\
	Y\\
	\end{bmatrix}_{(p+q)\times1}.$\\
	This gives $I(G)Y=\tau_i X$ and $(I(G))^TX=\tau_i Y.$ \\ Next to find the eigenvalues of $S(G)_k$.\\
	Let $Z=\begin{bmatrix}
	\sqrt{k}X\\
	Y\\
	Y\\
	\vdots\\
	Y
	\end{bmatrix}_{(p+kq)\times 1}$ be an eigenvector corresponding to a non-zero eigenvalue $\sqrt{k}\tau_i$, $1\leq i\leq 2r$ of $S(G)_k.$ This is because
	\begin{align*}
		A(S(G)_k)Z=&\begin{bmatrix}
	O & I(G) &  I(G) & \dots  & I(G)\\
	(I(G))^T & O & O & \dots  & O \\
	(I(G))^T & O & O & \dots  & O \\
	\vdots & \vdots & \vdots & \ddots & \vdots \\
	(I(G))^T & O & O & \dots  & O\\
	\end{bmatrix}_{p+kq}
	\begin{bmatrix}
	\sqrt{k}X\\
	Y\\
	Y\\
	\vdots\\
	Y
	\end{bmatrix}_{(p+kq)\times 1}\\=&\begin{bmatrix}
	 kI(G)Y\\
	\sqrt{k}(I(G))^TX\\
	\sqrt{k}(I(G))^TX\\
	\vdots\\
	\sqrt{k}(I(G))^TX
	\end{bmatrix}_{(p+kq)\times 1}\\
 	=&\sqrt{k}\begin{bmatrix}
	 \sqrt{k}I(G)Y\\
	\ (I(G))^TX\\
	\ (I(G))^TX\\
	\vdots\\
	\ (I(G))^TX
	\end{bmatrix}_{(p+kq)\times 1}\\
	=&\sqrt{k}\tau_i\begin{bmatrix}
	\sqrt{k}X\\
	\ Y\\
	\ Y\\
	\vdots\\
	\ Y
	\end{bmatrix}_{(p+kq)\times 1}\\
	=&\sqrt{k}\tau_i Z.
	\end{align*}
	Thus $\sqrt{k}\tau_i$'s,  $i=1,2,...,2r$ are non-zero eigenvalues of $S(G)_k$.  Therefore  $p+kq-2r$ eigenvalues of  $S(G)_k$ are zeros. Thus, \[spec(S(G)_k)=\begin{pmatrix}
	0 & \sqrt{k}\tau_1 & \sqrt{k}\tau_2 & \dots&\sqrt{k}\tau_{2r}\\
	p+kq-2r &1 & 1 & \dots & 1\\
	\end{pmatrix}.\]
	
\end{proof}
\begin{cor}
	Let $G$ be a simple $(p,q)$ graph. Then $\varepsilon(S(G)_k)=\sqrt{k}\varepsilon(S(G))$.
\end{cor}

\begin{cor}
	Let $G$ be $r_1$-regular graph.  Then $\varepsilon(S(G)_k)=2\sqrt{k}\sum_{i=1}^{p}\sqrt{\lambda_i+r_1}.$
\end{cor}

\begin{flushleft}In the following  corollary, we give a method to construct a new family of integral graphs.
	\end{flushleft} 
\begin{cor}\label{cor1}
	Let $G$ be a simple $(p,q)$ graph such that $S(G)$ is integral and $k$ be a perfect square. Then $S(G)_k$ is integral.
\end{cor}
\begin{exam}
	Let $G= K_{1,3}.$ Then $spec(S(G))= \begin{pmatrix}
	\ -2&-1&0&1&2\\
	\ 1&2&1&2&1
	\end{pmatrix}.$ Thus $S(G)_4,S(G)_9, \\S(G)_{16} $ etc. are integral graphs.
\end{exam}

\begin{thm}
	Let $G$  be a simple $(p,q)$ graph and $S(G)$ be the subdivision graph of $G$. Then the Randi{\'c} spectrum of $S(G)_k$ is,	\[RS(S(G)_k)=\begin{pmatrix}
	0 & \gamma_1 & \gamma_2 & \dots&\gamma_{2r}\\
	p+kq-2r &1 & 1 &\dots & 1\\
	\end{pmatrix}\] where $\gamma_i$'s, $1\leq i\leq 2r$  are the non-zero Randi{\'c} eigenvalues of $S(G)$ .
\end{thm}

\begin{proof}
	Let $M$ and $N$ be column matrix with orders $p\times1$ and $q\times 1$ respectively and let $V=\begin{bmatrix}
	M\\
	N\\
	\end{bmatrix}_{(p+q)\times1}$be the Randi{\'c} eigenvector corresponding to the non-zero Randi{\'c} eigenvalue $\gamma_i $,$1\leq i\leq 2r$  of $S(G).$ Then $R(S(G))V=\gamma_i V$.\\ That is,  
	\begin{align*}
	\begin{bmatrix}
	O&D^{-\frac{1}{2}}I(G)(2I_q)^{-\frac{1}{2}}\\
	(2I_q)^{-\frac{1}{2}}(I(G))^TD^{-\frac{1}{2}}&O\\
	\end{bmatrix}\begin{bmatrix}
	M\\
	N\\
	\end{bmatrix}=&\gamma_i\begin{bmatrix}
	M\\
	N\\
	\end{bmatrix}.
	\end{align*}
	This gives $D^{-\frac{1}{2}}I(G)(2I_q)^{-\frac{1}{2}}N=\gamma_i M$ and $(2I_q)^{-\frac{1}{2}}(I(G))^TD^{-\frac{1}{2}}M=\gamma_i N.$\\
	Next to find the Randi{\'c} eigenvalues of $S(G)_k$.\\	If $V=\begin{bmatrix}
	M\\
	N\\
	\end{bmatrix}_{(p+q)\times1}$ is the Randi{\'c} eigenvector of $S(G)$ corresponding to non-zero Randi{\'c} eigenvalue $\gamma_i $, $1\leq i\leq2r$, then  $W=\begin{bmatrix}
	\sqrt{k}M\\
	N\\
	N\\
	\vdots\\
	N
	\end{bmatrix}_{(p+kq)\times 1}$ is an eigenvector corresponding to non-zero Randi{\'c} eigenvalue $\gamma_i$ of $S(G)_k.$ This is because 
	\begin{align*} 
	  R(S(G)_k)W=&\left[\begin{smallmatrix}
		(kD)^{-\frac{1}{2}} & O &  O & \dots  & O\\
		O & (2I_q)^{-\frac{1}{2}} & O & \dots  & O \\
		O & O & (2I_q)^{-\frac{1}{2}} & \dots  & O \\
		\vdots & \vdots & \vdots & \ddots & \vdots \\
		O & O & O & \dots  & (2I_q)^{-\frac{1}{2}}\\
	\end{smallmatrix}\right]\cdot
	\left[\begin{smallmatrix}
		O & I(G) &  I(G) & \dots  & I(G)\\
		(I(G))^T & O & O & \dots  & O \\
		(I(G))^T & O & O & \dots  & O \\
		\vdots & \vdots & \vdots & \ddots & \vdots \\
		(I(G))^T & O &O & \dots  & O\\
	\end{smallmatrix}\right]\notag\\ & \hspace{5cm}
	\left[\begin{smallmatrix}
		(kD)^{-\frac{1}{2}} &O &  O & \dots  &O\\
		O & (2I_q)^{-\frac{1}{2}} & O & \dots  & O \\
		O & O & (2I_q)^{-\frac{1}{2}} & \dots  & O \\
		\vdots & \vdots & \vdots & \ddots & \vdots \\
		O & O & O & \dots  & (2I_q)^{-\frac{1}{2}}\\
	\end{smallmatrix}\right]\begin{bmatrix}
	\sqrt{k}M\\
	\ N\\
	\ N\\
	\vdots\\
	\ N
	\end{bmatrix}\notag\\
	 =& \left[\begin{smallmatrix}
		O& (kD)^{-\frac{1}{2}}I(G)(2I_q)^{-\frac{1}{2}} &  \dots  & (kD)^{-\frac{1}{2}}I(G)(2I_q)^{-\frac{1}{2}}\\
		(2I_q)^{-\frac{1}{2}}(I(G))^T(kD)^{-\frac{1}{2}} &  O & \dots  & O \\
		\vdots & \vdots &  \ddots & \vdots \\
		(2I_q)^{-\frac{1}{2}}(I(G))^T(kD)^{-\frac{1}{2}} &  O & \dots  & O\\
	\end{smallmatrix}\right]\begin{bmatrix}
	\sqrt{k}M\\
	\ N\\
	\ N\\
	\vdots\\
	\ N
	\end{bmatrix}\notag\\
	=&\begin{bmatrix}
	\ k(kD)^{-\frac{1}{2}}I(G)(2I_q)^{-\frac{1}{2}}N\\
	\ \sqrt{k}(2I_q)^{-\frac{1}{2}}(I(G))^T(kD)^{-\frac{1}{2}}M\\
	\ \sqrt{k}(2I_q)^{-\frac{1}{2}}(I(G))^T(kD)^{-\frac{1}{2}}M\\
	\vdots\\
	\ \sqrt{k}(2I_q)^{-\frac{1}{2}}(I(G))^T(kD)^{-\frac{1}{2}}M
	\end{bmatrix}=\begin{bmatrix}
	\sqrt{k}\gamma_i M\\
	\ \gamma_i N\\
	\ \gamma_i N\\
	\vdots\\
	\ \gamma_i N
	\end{bmatrix}=\gamma_i \begin{bmatrix}
	\sqrt{k}M\\
	\ N\\
	\ N\\
	\vdots\\
	\ N
	\end{bmatrix}=\gamma_i W.
	\end{align*}
	Thus $\gamma_i$'s ,$i=1,2,...,2r$ are non-zero Randi{\'c} eigenvalues of $S(G)_k$. Therefore $p+kq-2r$ Randi{\'c} eigenvalues of $R(S(G)_k)$ are zeros. Thus \[RS(S(G)_k)=\begin{pmatrix}
	0 & \gamma_1 & \gamma_2 & \dots&\gamma_{2r}\\
	p+kq-2r &1 & 1 & \dots & 1\\
	\end{pmatrix}.\]
\end{proof}
\begin{flushleft}

In the following corollary, we give a method to construct a pair of graphs having the same Randi{\'c} energy.
\end{flushleft}
\begin{cor}\label{co6}
	Let $G$ be a simple $(p,q)$ graph. Then $\varepsilon_R(S(G)_k) = \varepsilon_R(S(G)).$
	\end{cor}
We shall introduce the following operations on $G$.

\begin{op}\label{op2}
\par Let $G$ be a simple $(p,q)$ graph with vertex set $V(G)=\{v_{1}, v_{2},...,v_{p}\}$, edge set $E(G)=\{e_{1}, e_{2},...,e_{q}\}$ and $S(G)$  be the subdivision graph of $G$ with vertex set $V(G)\bigcup E(G)$. Corresponding to each vertex $v_j$,$1\leq j\leq p$ in $S(G)$, introduce a set $V_j^n$ of $n$ isolated vertices and join each vertex of $V_j^n$ to the neighbors of $v_j$  in $S(G)$. Then  corresponding to each vertex $e_j$ in $S(G)$ introduce a set of $t$ isolated vertices $E_j^t$ ,$1\leq j\leq q$, where $t=n$ or $t=n-1$, and  join each vertex in $E_j^t$ to the neighbors of $e_j,1\leq j\leq q$. The resultant graph is denoted by $S(G)_t^n$.
\end{op}Note that, in the graph $S(G)_t^n$ has $(n+1)p+(t+1)q$ vertices
and $(t+1)(n+1)2q$ edges.\\
The following figures illustrates the above operation	
\begin{figure}[H]
	\begin{minipage}[b]{0.5\linewidth}
		\centering
		\includegraphics[width=4.0cm]{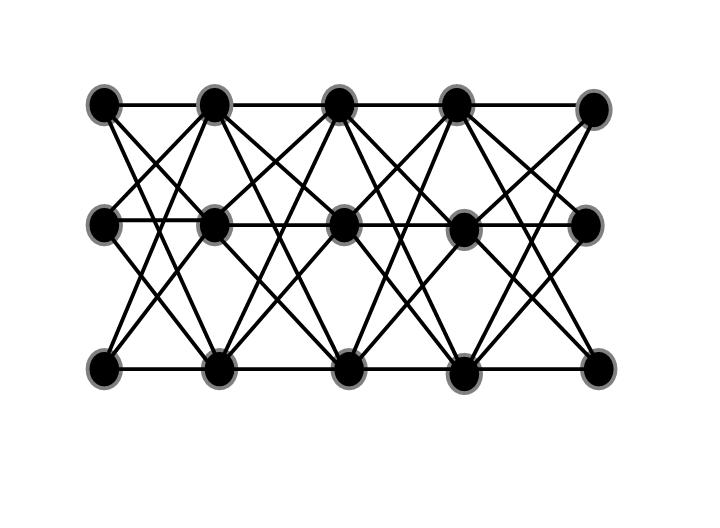}
		\caption{$S(P_3)_2^2$}
		\label{pict21.jpg}
	\end{minipage}
	\hspace{0.5cm}
	\begin{minipage}[b]{0.5\linewidth}
		\centering
		\includegraphics[width=4.0cm]{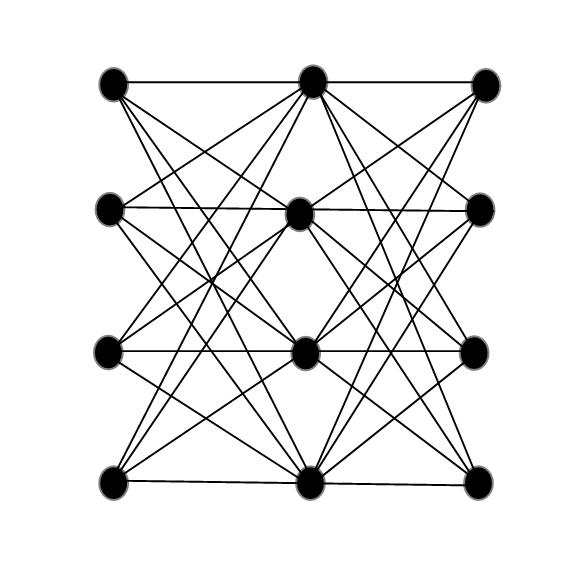}
		\caption{$S(K_2)_3^3$}
		\label{pict22.jpg}
	\end{minipage}
\end{figure} 
\begin{figure}[H]
	\centering
	\includegraphics[width=4.0cm]{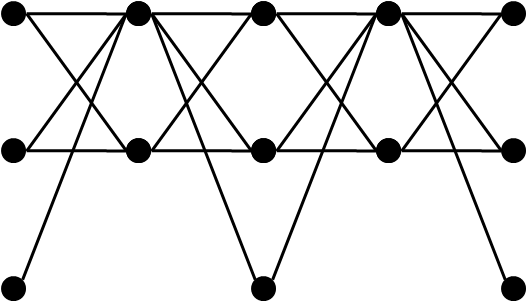}
	\caption{$S(P_3)_1^2$}
	\label{pict26.jpg}
\end{figure} 
 
 Case 1. If $t=n$, then by proper labeling of the vertices of $S(G)_t^n$, it can be easily seen that   
\[ A(S(G)_t^n) =\begin{bmatrix}
O & I(G) &  O & \dots  &O&I(G)\\
(I(G))^T & O & (I(G))^T & \dots  &(I(G))^T &O \\
O & I(G) &  O & \dots  &O&I(G)\\
\vdots & \vdots & \vdots & \ddots & \vdots \\
O & I(G) &  O & \dots  &O&I(G)\\
(I(G))^T & O & (I(G))^T & \dots  & (I(G))^T&O
\end{bmatrix}.\]
Case 2. If $t=n-1$, then by  proper labeling of the vertices of $S(G)_t^n$, it can be easily seen that  \[ A(S(G)_t^n) =\begin{bmatrix}
O & I(G) & O & \dots &I(G) &O\\
(I(G))^T & O & (I(G))^T & \dots  &O& (I(G))^T \\
O & I(G) &  O & \dots  &I(G)&O\\
\vdots & \vdots & \vdots & \ddots & \vdots \\
(I(G))^T & O & (I(G))^T & \dots  &O& (I(G))^T \\
O &I(G) & O & \dots  & I(G)&O
\end{bmatrix}.\]  \\\\\\
	The degree matrix of $S(G)_t^n$ is as follows \\ If $t=n$, then \[D(S(G)_t^n)=\begin{bmatrix}
tD & O & O & \dots  & O\\
O & 2nI_q & O & \dots  & O \\
O & O & tD & \dots  & O \\
\vdots & \vdots & \vdots & \ddots & \vdots \\
O & O& O & \dots  & 2nI_q\\
\end{bmatrix}.\]\\
If $t=n-1$, then \[D(S(G)_t^n)=\begin{bmatrix}
tD & O & O & \dots  & O\\
O & 2nI_q & O & \dots  & O \\
O & O & tD & \dots  & O \\
\vdots & \vdots & \vdots & \ddots & \vdots \\
O & O& O & \dots  & tD\\
\end{bmatrix}.\]\\
The Randi{\'c} matrix of $S(G)_t^n$ is as \\If $t=n$, then
$$R(S(G)_t^n)=\left[\begin{smallmatrix}
O& (tD)^{-\frac{1}{2}}I(G)(2nI_q)^{-\frac{1}{2}} &  O & \dots  & (tD)^{-\frac{1}{2}}I(G)(2nI_q)^{-\frac{1}{2}}\\
(2nI_q)^{-\frac{1}{2}}(I(G))^T(tD)^{-\frac{1}{2}} &O & (2nI_q)^{-\frac{1}{2}}(I(G))^T(tD)^{-\frac{1}{2}} & \dots  & O \\
O& (tD)^{-\frac{1}{2}}I(G)(2nI_q)^{-\frac{1}{2}} &  . & \dots  & (tD)^{-\frac{1}{2}}I(G)(2nI_q)^{-\frac{1}{2}}\\
\vdots & \vdots & \vdots & \ddots & \vdots \\
(2nI_q)^{-\frac{1}{2}}(I(G))^T(tD)^{-\frac{1}{2}} & O & (2nI_q)^{-\frac{1}{2}}(I(G))^T(tD)^{-\frac{1}{2}} & \dots  & O\\
\end{smallmatrix}\right].$$
\\If $t=n-1$, then
$$R(S(G)_t^n)=\left[\begin{smallmatrix}
O& (tD)^{-\frac{1}{2}}I(G)(2nI_q)^{-\frac{1}{2}} &  O & \dots  & O\\
(2nI_q)^{-\frac{1}{2}}(I(G))^T(tD)^{-\frac{1}{2}} &O & (2nI_q)^{-\frac{1}{2}}(I(G))^T(tD)^{-\frac{1}{2}} & \dots  & (2nI_q)^{-\frac{1}{2}}(I(G))^T(tD)^{-\frac{1}{2}} \\
O& (tD)^{-\frac{1}{2}}I(G)(2nI_q)^{-\frac{1}{2}} &  . & \dots  & O\\
\vdots & \vdots & \vdots & \ddots & \vdots \\
(2nI_q)^{-\frac{1}{2}}(I(G))^T(tD)^{-\frac{1}{2}} &O & (2nI_q)^{-\frac{1}{2}}(I(G))^T(tD)^{-\frac{1}{2}} & \dots  & (2nI_q)^{-\frac{1}{2}}(I(G))^T(tD)^{-\frac{1}{2}} \\
O & (2nI_q)^{-\frac{1}{2}}(I(G))^T(tD)^{-\frac{1}{2}} & O& \dots  & O\\
\end{smallmatrix}\right].$$

\begin{thm}\label{th7}
		Let $G$  be a simple $(p,q)$ graph and $S(G)$ be the subdivision graph of $G$. Then the spectrum of $S(G)_t^n$ is	\[spec(S(G)_t^n)=\left(\begin{smallmatrix}
	      0 & \sqrt{(t+1)(n+1)}\tau_1 & \sqrt{(t+1)(n+1)}\tau_2 & \dots&\sqrt{(t+1)(n+1)}\tau_{2r}\\
	       (n+1)p+(t+1)q-2r &1 & 1 &\dots & 1
	\end{smallmatrix}\right)\] where $\tau_i$'s,$1\leq i\leq 2r$ are non-zero  eigenvalues of $S(G)$.
\end{thm}
\begin{proof}
	Let $X$ and $Y$ be column matrix with orders $p\times1$ and $q\times 1$ respectively and $U=\begin{bmatrix}
	X\\
	Y\\
	\end{bmatrix}_{(p+q)\times1}$ be the eigenvector corresponding to the non-zero  eigenvalue $\tau_i$, $1\leq i\leq 2r,$  of $S(G).$ Then $A(S(G))U=\tau_i U.$ \\That is,
	$\begin{bmatrix}
	O&I(G)\\
	(I(G))^T&O\\
	\end{bmatrix}_{(p+q)\times(p+q)}\begin{bmatrix}
	X\\
	Y\\
	\end{bmatrix}_{(p+q)\times1}=\tau_i\begin{bmatrix}
	X\\
	Y\\
	\end{bmatrix}_{(p+q)\times1}$.\\
	This gives $I(G)Y=\tau_i X$ and $(I(G))^TX=\tau_i Y.$\\
Next to find the eigenvalues of $S(G)_t^n$.
 \\Case 1. $t=n.$\\	Let $Z_1=\begin{bmatrix}
X\\
Y\\
\\
\vdots\\
X\\
Y
\end{bmatrix}_{((n+1)p+(t+1)q)\times 1}$ be an eigenvector corresponding to the non-zero eigenvalue $(n+1)\tau_i$, $1\leq i\leq 2r$  of $S(G)_t^n.$ This is because  
\begin{align*}
 A(S(G)_t^n)Z_1=\begin{bmatrix}
O & I(G) &  O & \dots  &I(G)\\
(I(G))^T & O & (I(G))^T & \dots  & O \\
O & I(G) &  O & \dots  &I(G)\\
\vdots & \vdots & \vdots & \ddots & \vdots \\
(I(G))^T &O & (I(G))^T & \dots  & O
\end{bmatrix}\begin{bmatrix}
\ X\\
\ Y\\
\vdots\\
\ X\\
\ Y
\end{bmatrix}=&\begin{bmatrix}
\ (n+1)\tau_i X\\
\ (n+1)\tau_i Y\\
\vdots\\
\ (n+1)\tau_i X\\
\ (n+1)\tau_i Y
\end{bmatrix}\\=&(n+1)\tau_i\begin{bmatrix}
\ X\\
\ Y\\
\vdots\\
\ X\\
\ Y
\end{bmatrix}\\=&(n+1)\tau_i Z_1.\\ \end{align*}
Case 2. $t=n-1.$\\ Let $Z_2=\begin{bmatrix}
	\sqrt{(t+1)} X\\
	\sqrt{(n+1)} Y\\
	\vdots\\
	\sqrt{(t+1)} X\\
	\sqrt{(n+1)} Y\\
	\sqrt{(t+1)} X
\end{bmatrix}_{((n+1)p+(t+1)q)\times 1}$ be an eigenvector corresponding to the non-zero eigenvalue $\sqrt{(n+1)(t+1)}\tau_i$, $i\leq i \leq 2r$  of $S(G)_t^n.$ This is because  \begin{align*} A(S(G)_t^n)Z_2=&\begin{bmatrix}
O & I(G) &  O & \dots  &I(G)&O\\
(I(G))^T & O & (I(G))^T & \dots  & O&(I(G)^T) \\
O & I(G) &  O & \dots  &I(G)&O\\
\vdots & \vdots & \vdots & \ddots & \vdots \\
(I(G))^T & O & (I(G))^T & \dots  & O&(I(G)^T) \\
O & I(G) & O & \dots  & I(G)&O
\end{bmatrix}
\begin{bmatrix}
\sqrt{(t+1)} X\\
\sqrt{(n+1)} Y\\
\vdots\\
\sqrt{(t+1)} X\\
\sqrt{(n+1)} Y\\
\sqrt{(t+1)} X
\end{bmatrix}\\=&\begin{bmatrix}
\ (t+1)\sqrt{(n+1)}\tau_i X\\
\ (n+1)\sqrt{(t+1)}\tau_i Y\\
\vdots\\
\ (t+1)\sqrt{(n+1)}\tau_i X\\
\ (n+1)\sqrt{(t+1)}\tau_i Y\\
\ (t+1)\sqrt{(n+1)}\tau_i X
\end{bmatrix}\\
=&\sqrt{(n+1)(t+1)}\tau_i\begin{bmatrix}
\sqrt{(t+1)} X\\
\sqrt{(n+1)} Y\\
\vdots\\
\sqrt{(t+1)} X\\
\sqrt{(n+1)} Y\\
\sqrt{(t+1)} X
\end{bmatrix}\\=&\sqrt{(n+1)(t+1)}\tau_i Z_2.
\end{align*}
Thus, if $\tau_i $'s, $1\leq i\leq 2r,$ are non-zero eigenvalue of $S(G)$, then $\sqrt{(n+1)(t+1)}\tau_i$ is  non-zero eigenvalue of $S(G)_t^n.$ Therefore $(n+1)p+(t+1)q-2r$  eigenvalues of $S(G)_t^n$ are zeros. Thus, \[spec(S(G)_t^n)=\left(\begin{smallmatrix}
0 & \sqrt{(t+1)(n+1)}  \tau_1 & \sqrt{(t+1)(n+1)} \tau_2 & \dots&\sqrt{(t+1)(n+1)} \tau_{2r}\\
(n+1)p+(t+1)q-2r &1 & 1 &\dots & 1
\end{smallmatrix}\right).\]
\end{proof}


\begin{cor}
	The energy of $S(G)_t^n$ is, $\varepsilon(S(G)_t^n)=\sqrt{(n+1)(t+1)}\varepsilon(S(G)).$
\end{cor}
\begin{exam}
	Let $G= K_2$. Then \[spec(S(K_2))=\begin{pmatrix}
	-\sqrt{2}&0&\sqrt{2}\\
	1&1&1\\
	\end{pmatrix}.\] 
	If $n=2$ and $t=1$, then the adjacency matrix of $S(K_2)_t^n$ is, \[A(S(K_2))_1^2=\begin{pmatrix}
	0&0&1&0&0&1&0&0\\
	0&0&1&0&0&1&0&0\\
	1&1&0&1&1&0&1&1\\
	0&0&1&0&0&1&0&0\\
	0&0&1&0&0&1&0&0\\
	1&1&0&1&1&0&1&1\\
	0&0&1&0&0&1&0&0\\
	0&0&1&0&0&1&0&0
	\end{pmatrix}.\] 
	\[spec(A(S(K_2))_1^2)=\begin{pmatrix}
	-2\sqrt{3}&0&2\sqrt{3}\\
	1&6&1\\
	\end{pmatrix}.\] 
Then $\varepsilon(S(K_2)_1^2)=\sqrt{6}\varepsilon(S(K_2)).$	
\end{exam} .
\begin{exam}
	Let $G=K_{1,3}$, then\[spec(S(K_{1,3}))=\begin{pmatrix}
-2&-1&0&1&2\\
1&2&1&2&1\\
\end{pmatrix}.\]
\[spec(S(K_{1,3})_1^1)=\begin{pmatrix}
-4&-2&0&2&4\\
1&2&8&2&1\\
\end{pmatrix}.\]
 Hence $\varepsilon(S(K_{1,3})_1^1)=2\varepsilon(S(K_{1,3})).$
\end{exam}
\begin{rem}
	The  spectrum of $S(G)_{(t+1)(n+1)}$ is\[spec(S(G)_{(t+1)(n+1)})=\left(\begin{smallmatrix}
	0 & \sqrt{(t+1)(n+1)}\tau_1 & \sqrt{(t+1)(n+1)}\tau_2 & \dots&\sqrt{(t+1)(n+1)}\tau_{2r}\\
	p+(n+1)(t+1)q-2r &1 & 1 & \dots & 1\\
	\end{smallmatrix}\right).\] The  spectrum of $S(G)_t^n$ is\[spec(S(G)_t^n)=\left(\begin{smallmatrix}
	0 & \sqrt{(t+1)(n+1)}\tau_1 & \sqrt{(t+1)(n+1)}\tau_2 & \dots&\sqrt{(t+1)(n+1)}\tau_{2r}\\
	(n+1)p+(t+1)q-2r &1 & 1 &\dots & 1
	\end{smallmatrix}\right).\]  The energy of the graphs $S(G)_{(t+1)(n+1)}$ and $S(G)_t^n$ are same  but their orders are different.
\end{rem}
\begin{flushleft}Now we obtain a new family of integral graphs as follows.
\end{flushleft}
\begin{cor}
		If $t=n$ and $S(G)$ is integral, then $S(G)_t^n$ is integral.
\end{cor}
\begin{exam}
	Let $G= K_{1,3}$, $spec(S(G))= \begin{pmatrix}
	\ -2&-1&0&1&2\\
	\ 1&2&1&2&1
	\end{pmatrix}$. Then $S(G)_n^n $ is integral for every $n$.
	\end{exam}
In the following Proposition gives some pairs of cospectral graphs.
\begin{prop}
	Let $G$ be a simple $(p,q)$ graph with $p=q$, then the graphs $S(G)_2$ and $S(G)_0^1$ are cospectral graphs.
\end{prop}
\begin{proof}
	Proof follows from Theorems \ref{th1} and \ref{th7}.

\end{proof}
\begin{exam}
\end{exam}
Consider the graphs $G$(see Figure 5). Then the spectrum of $S(G),S(G)_2$ and $S(G)_0^1$  are
\begin{figure}[H]
\centering
\includegraphics[width=4.0cm]{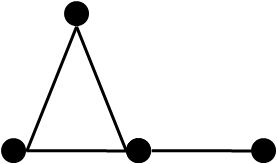}
\caption{Graph $G$ with $p=q$}
\label{pict26.jpg}
\end{figure} 
\begin{flushleft}
\end{flushleft}
\[spec(S(G))=\begin{pmatrix}
\sqrt{\frac{5+\sqrt{17}}{2}}&-\sqrt{\frac{5+\sqrt{17}}{2}}&\sqrt{2}&-\sqrt{2}&-1&1&\sqrt{\frac{5-\sqrt(17)}{2}}&-\sqrt{\frac{5-\sqrt(17)}{2}}\\
1&1&1&1&1&1&1&1\\
\end{pmatrix}.\] \[spec(S(G)_2)=\begin{pmatrix}
0&\sqrt{5+\sqrt{17}}&-\sqrt{5+\sqrt{17}}&2&-2&-\sqrt{2}&\sqrt{2}&\sqrt{5-\sqrt{17}}&-\sqrt{5-\sqrt{17}}\\
4&1&1&1&1&1&1&1&1\\
\end{pmatrix}.\]
\[spec(S(G)_0^1)=\begin{pmatrix}
0&\sqrt{5+\sqrt{17}}&-\sqrt{5+\sqrt{17}}&2&-2&-\sqrt{2}&\sqrt{2}&\sqrt{5-\sqrt{17}}&-\sqrt{5-\sqrt{17}}\\
4&1&1&1&1&1&1&1&1\\
\end{pmatrix}.\]
\begin{rem}
Let $G$  be a simple $(p,q)$ graph and $S(G)$ be the subdivision graph of $G$. Then $\varepsilon(S(G)_0^1) =\varepsilon(S(G)_2)$.\\ 
\end{rem}
We shall now discuss the problem of constructing pairs of graphs having the same Randi{\'c} energy.
\begin{thm}\label{th9}
Let $G$ be a simple $(p,q)$ graph. Then  $\varepsilon_R(S(G)_t^n)=\varepsilon_R(S(G))$ .  
\end{thm}
\begin{proof}
Let $M$ and $N$ be column matrix with orders $p\times1$ and $q\times 1$ respectively and let $V=\begin{bmatrix}
M\\
N\\
\end{bmatrix}_{(p+q)\times1}$be the Randi{\'c} eigenvector corresponding to the non-zero Randi{\'c} eigenvalue $\gamma_i $, $1\leq i\leq 2r,$ of $S(G).$ Then $R(S(G)_t^n)V=\gamma_i V$.\\  That is,
\begin{align*}
\begin{bmatrix}
O&D^{-\frac{1}{2}}I(G)(2I_q)^{-\frac{1}{2}}\\
(2I_q)^{-\frac{1}{2}}(I(G))^TD^{-\frac{1}{2}}&O\\
\end{bmatrix}\begin{bmatrix}
M\\
N\\
\end{bmatrix}=&\gamma_i\begin{bmatrix}
M\\
N\\
\end{bmatrix}.
\end{align*}
This gives $D^{-\frac{1}{2}}I(G)(2I_q)^{-\frac{1}{2}}N=\gamma_i M$ and $(2I_q)^{-\frac{1}{2}}(I(G))^TD^{-\frac{1}{2}}M=\gamma_i N.$\\
	Next, to find the eigenvalues of $R(S(G)_t^n)$.\\ Case 1.  $t=n.$\\	If $V$ is the Randi{\'c} eigenvector of $S(G)$ corresponding to non-zero Randi{\'c} eigenvalue $\gamma_i $, $1\leq i\leq 2r,$ then $W_1=\begin{bmatrix}
	M\\
	N\\
	M\\
	\vdots\\
	M\\
	N
	\end{bmatrix}_{((n+1)p+(t+1)q)\times 1}$ is a  Randi{\'c} eigenvector corresponding to non-zero Randi{\'c}  eigenvalue $\gamma_i$ of $S(G)_t^n.$ This is because 

\begin{align*} 
&R(S(G)_t^n)W_1=\left[\begin{smallmatrix}
(nD)^{-\frac{1}{2}} & O & O & \dots  & O\\
O & (2nI_q)^{-\frac{1}{2}} & O & \dots  & O \\
O & O & (2nI_q)^{-\frac{1}{2}} & \dots  & O \\
\vdots & \vdots & \vdots & \ddots & \vdots \\
O & O& O & \dots  & (2nI_q)^{-\frac{1}{2}}\\
\end{smallmatrix}\right]\cdot
\left[\begin{smallmatrix}
O & I(G) &  O & \dots  &I(G)\\
(I(G))^T & O & I(G)^T & \dots  &O \\
O & I(G) & O& \dots  &I(G)\\
\vdots & \vdots & \vdots & \ddots & \vdots \\
(I(G))^T & O & (I(G))^T & \dots  &O
\end{smallmatrix}\right]\notag\\ & \hspace{5cm}
\left[\begin{smallmatrix}
(nD)^{-\frac{1}{2}} & O &  O & \dots  & O\\
O & (2nI_q)^{-\frac{1}{2}} & O & \dots  & O \\
O & O & (2nI_q)^{-\frac{1}{2}} & \dots  & O \\
\vdots & \vdots & \vdots & \ddots & \vdots \\
O & O & O & \dots  & (2nI_q)^{-\frac{1}{2}}\\
\end{smallmatrix}\right]\begin{bmatrix}
\ M\\
\ N\\
\ M\\
\vdots\\
\ M\\
\ N
\end{bmatrix}\notag\end{align*} \begin{align*}
= &\left[\begin{smallmatrix}
O& (nD)^{-\frac{1}{2}}I(G)(2nI_q)^{-\frac{1}{2}} &  . & \dots  & (nD)^{-\frac{1}{2}}I(G)(2nI_q)^{-\frac{1}{2}}\\
(2nI_q)^{-\frac{1}{2}}(I(G))^T(nD)^{-\frac{1}{2}} & O & (2nI_q)^{-\frac{1}{2}}(I(G))^T(nD)^{-\frac{1}{2}} & \dots  & O \\
O& (nD)^{-\frac{1}{2}}I(G)(2nI_q)^{-\frac{1}{2}} &  . & \dots  & (nD)^{-\frac{1}{2}}I(G)(2nI_q)^{-\frac{1}{2}}\\
\vdots & \vdots & \vdots & \ddots & \vdots \\
(2nI_q)^{-\frac{1}{2}}(I(G))^T(nD)^{-\frac{1}{2}} & O & (2nI_q)^{-\frac{1}{2}}(I(G))^T(nD)^{-\frac{1}{2}} & \dots  & O\\
\end{smallmatrix}\right]\begin{bmatrix}
\ M\\
\ N\\
\ M\\
\vdots\\
\ M\\
\ N
\end{bmatrix}\notag\\ 
=&\begin{bmatrix}
\ n(nD)^{-\frac{1}{2}}I(G)(2nI_q)^{-\frac{1}{2}}N\\
\ n(2nI_q)^{-\frac{1}{2}}(I(G))^T(nD)^{-\frac{1}{2}}M\\
\ n(2nI_q)^{-\frac{1}{2}}I(G)(nD)^{-\frac{1}{2}}N\\
\vdots\\
\ n(2nI_q)^{-\frac{1}{2}}(I(G))^T(nD)^{-\frac{1}{2}}M
\end{bmatrix}=\begin{bmatrix}
\gamma_i M\\
\ \gamma_i N\\
\ \gamma_i M\\
\vdots\\
\ \gamma_i N
\end{bmatrix}=\gamma_i \begin{bmatrix}
 \ M \\
\ N\\
\ M\\
\vdots\\
\ M\\
\ N
\end{bmatrix}=\gamma_i W_1.
\end{align*}
	Case 2. $t=n-1.$ \\If $V$ is the Randi{\'c} eigenvector of $S(G)$ corresponding to non-zero Randi{\'c} eigenvalue $\gamma_i $, $1\leq i\leq 2r,$ then $W_2=\begin{bmatrix}
	\sqrt{t}M\\
	\sqrt{n}N\\
	\sqrt{t}M\\
	\vdots\\
	\sqrt{t}M\\
	\sqrt{n}N\\
	\sqrt{t}M
	\end{bmatrix}_{((n+1)p+(t+1)q)\times 1}$ is a Randi{\'c} eigenvector corresponding to non-zero Randi{\'c}  eigenvalue $\gamma_i$ of $S(G)_t^n.$ This is because 
	\begin{align*} 
	R(S(G)_t^n)W_2=&\left[\begin{smallmatrix}
	((n-1)D)^{-\frac{1}{2}} & O &  O & \dots  & O\\
	O & (2nI_q)^{-\frac{1}{2}} & O & \dots  & O \\
    O & O & ((n-1)D)^{-\frac{1}{2}} & \dots  & O\\
	\vdots & \vdots & \vdots & \ddots & \vdots \\
	O & O & O & \dots  & ((n-1)D)^{-\frac{1}{2}}\\
	\end{smallmatrix}\right]\cdot
	\left[\begin{smallmatrix}
	O & I(G) &  O & \dots  &I(G)&O\\
	(I(G))^T & 0 & (I(G))^T & \dots  & O &(I(G))^T\\
	O & I(G) &  O & \dots  &I(G)&O\\
	\vdots & \vdots & \vdots & \ddots & \vdots \\
	(I(G))^T & 0 & (I(G))^T & \dots  & O &(I(G))^T\\
	O &I(G) & O & \dots  & I(G)&O
	\end{smallmatrix}\right]\notag\\ & \hspace{4cm}
	\left[\begin{smallmatrix}
	((n-1)D)^{-\frac{1}{2}} & O &  O & \dots  & O\\
	O & (2nI_q)^{-\frac{1}{2}} & O & \dots  & O \\
	O & O & ((n-1)D)^{-\frac{1}{2}} & \dots  & O\\
	\vdots & \vdots & \vdots & \ddots & \vdots \\
	O & O & O & \dots  & ((n-1)D)^{-\frac{1}{2}}\\
	\end{smallmatrix}\right]\begin{bmatrix}
	\ \sqrt{t}M\\
	\ \sqrt{n}N\\
	\ \sqrt{t}M\\
	\vdots\\
	\ \sqrt{t}M\\
	\ \sqrt{n}N\\
	\ \sqrt{t}M
	\end{bmatrix}\notag 
	\end{align*}\begin{align*}
	= &\left[\begin{smallmatrix}
	O& (tD)^{-\frac{1}{2}}I(G)(2nI_q)^{-\frac{1}{2}} &  O & \dots  & O\\
	(2nI_q)^{-\frac{1}{2}}(I(G))^T(tD)^{-\frac{1}{2}} &O & (2nI_q)^{-\frac{1}{2}}(I(G))^T(tD)^{-\frac{1}{2}} & \dots  & (2nI_q)^{-\frac{1}{2}}(I(G))^T(tD)^{-\frac{1}{2}} \\
	O& (tD)^{-\frac{1}{2}}I(G)(2nI_q)^{-\frac{1}{2}} &  . & \dots  & O\\
	\vdots & \vdots & \vdots & \ddots & \vdots \\
	(2nI_q)^{-\frac{1}{2}}(I(G))^T(tD)^{-\frac{1}{2}} & O & (2nI_q)^{-\frac{1}{2}}(I(G))^T(tD)^{-\frac{1}{2}} & \dots  & (2nI_q)^{-\frac{1}{2}}(I(G))^T(tD)^{-\frac{1}{2}}\\
	O& (tD)^{-\frac{1}{2}}I(G)(2nI_q)^{-\frac{1}{2}} &  O & \dots  & O
	\end{smallmatrix}\right]\begin{bmatrix}
	\sqrt{t}M\\
	\ \sqrt{n}N\\
	\ \sqrt{t}M\\
	\vdots\\
	\ \sqrt{n}N\\
	\sqrt{t}M
	\end{bmatrix}\notag\\
	=&\begin{bmatrix}
	\ t(tD)^{-\frac{1}{2}}I(G)(2nI_q)^{-\frac{1}{2}}N\\
	\ n(2nI_q)^{-\frac{1}{2}}(I(G))^T(tD)^{-\frac{1}{2}}M\\
	\ t(tD)^{-\frac{1}{2}}I(G)(2nI_q)^{-\frac{1}{2}}N\\
	\vdots\\
	\ n(2nI_q)^{-\frac{1}{2}}(I(G))^T(tD)^{-\frac{1}{2}}M\\
	\ t(tD)^{-\frac{1}{2}}I(G)(2nI_q)^{-\frac{1}{2}}N
	\end{bmatrix}=\begin{bmatrix}
	\gamma_i \sqrt{t}M\\
	\ \gamma_i \sqrt{n}N\\
	\ \gamma_i \sqrt{t}M\\
	\vdots\\
	\ \gamma_i \sqrt{n}N\\
	\gamma_i \sqrt{t}M
	\end{bmatrix}=\gamma_i \begin{bmatrix}
	\ \sqrt{t}M\\
	\ \sqrt{n}N\\
	\ \sqrt{t}M\\
	\vdots\\
	\ \sqrt{n}N\\
	\ \sqrt{t}M
	\end{bmatrix}=\gamma_i W_2.
	\end{align*}
	Thus, if $\gamma_i$'s, $1\leq i\leq 2r,$ are non-zero Randi{\'c} eigenvalue of $S(G)$, then $\gamma_i$ is also non-zero Randi{\'c} eigenvalue of $S(G)_t^n.$ Therefore $(n+1)p+(t+1)q-2r$  Randi{\'c} eigenvalues $S(G)_t^n$ are zeros. Thus
	\[RS(S(G)_t^n)=\begin{pmatrix}
	0 & \gamma_1 & \gamma_2 & \dots&\gamma_{2r}\\
	(n+1)p+(t+1)q-2r &1 & 1 & \dots & 1\\
	\end{pmatrix}.\]
	 
\end{proof}

Next we give the construction of Randi{\'c} equienergetic graphs by means of graphs $S(G)_k$ and $S(G)_t^n$.
\begin{thm}
Let $G$ be a simple $(p,q)$ graph and $S(G)$  be the subdivision graph of $G.$  Then $S(G)_k$ and $S(G)_t^n$ are Randi{\'c} equienergetic whenever the order of $S(G)_k$ and $S(G)_t^n$ are  equal.
\end{thm}
\begin{proof}
 Proof follows from Corollary \ref{co6} and Theorem \ref{th9}  .
\end{proof}
\noindent\textbf{Note.} Let $G$ be a simple $(p,q)$ graph. 
\\(1) If $p=q,$ then the order of $S(G)_k$ and $S(G)_t^n$ are equal if and only if $k=n+t+1.$
(2) Order of $S(G)_k$ and $S(G)_t^n$ are equal if and only if $p+kq=(t+1)q+(n+1)p$\\ if and only if either $k=n(\frac{p+q}{q})+1$ or $k=n(\frac{p+q}{q}).$
	
\noindent From the following theorem, we construct an infinite family of Randi{\'c} cospectral graphs.

\begin{prop}
	Let $G$ be a simple $(p,q)$ graph,  $G_1=S(G)_k$ and $G_2=S(G)_t^n$ with $k=n\frac{p}{q}+t+1.$ Then $G_1$ and $G_2$ are Randi{\'c} cospectral.
\begin{proof}
	 By Theorem 3.5 and proof of Theorem 3.10, we get $RS(G_1)= RS(G_2)$ .
\end{proof}	
\end{prop}
\begin{exam}
Let $G=P_3$,  then 
\[RS(S(P_3))=\begin{pmatrix}
-1&-\frac{1}{\sqrt{2}}&0&\frac{1}{\sqrt{2}}&1\\
1&1&1&1&1\\
\end{pmatrix},\] \[RS(S(P_3)_2^2)=\begin{pmatrix}
-1&-\frac{1}{\sqrt{2}}&0&\frac{1}{\sqrt{2}}&1\\
1&1&11&1&1\\
\end{pmatrix},\] and  \[RS(S(P_3)_6) =\begin{pmatrix}
-1&-\frac{1}{\sqrt{2}}&0&\frac{1}{\sqrt{2}}&1\\
1&1&11&1&1\\
\end{pmatrix}.\]
Thus $RS(S(P_3)_2^2))=RS(S(P_3)_6)$. Note that $S(P_3)_2^2$ and $S(P_3)_6$ are non-isomorphic. Thus the graphs $S(P_3)_2^2$ and $S(P_3)_6$ are non-isomorphic Randi{\'c} cospectral graphs.
\end{exam}
\noindent\textbf{Note.}
The family of graphs $S(K_2)_t^n$ is Randi{\'c} integral for every $t$ and $n$ with $t=n$ or $t=n-1$. 
%
%
\section{Equienergetic  and Randi{\'c} equienergetic graphs }
In this section, we obtain an infinite  family of  equienergetic and Randi{\'c} equienergetic non-cospectral graphs.  
\begin{op}\label{dh3}
	Let $G$ be a simple $(p,q)$ graph, $D_m(G)$,$m>3$ be the $m$-shadow graph of $G$ and $G_1, G_2,...,G_m$ be the $ m$ copies of $G$ in $D_m(G)$. The graph $F_1^m(G)$ is obtained by deleting all edges connecting the vertices of $G_i$ and $G_{m-(i-1)},2\leq i\leq m-1$  and removing the edges of $G_1$.
\end{op}
The number of vertices and edges in  $F_1^m(G)$ are $pm$ and  $(m^2-m+1)q$ respectively.
\begin{thm}\label{h111}
	The energy of the graph $F_1^m(G)$  is, $\varepsilon(F_1^m(G)) = \bigg[m-2+\sqrt{m^2-2m+5}\bigg]\varepsilon(G) $.
\end{thm}
\begin{proof}
	With the suitable labeling of the vertices, the adjacency matrix of $F_1^m(G)$ is 
	\begin{align}
	A(F_1^m(G)) =&\begin{bmatrix}
	O &A(G) &  A(G) & \dots &A(G)& A(G)\\
	A(G) & A(G) &A(G) & \dots  &O& A(G) \\
	\vdots & \vdots & \vdots & \dots & \vdots & \vdots \\
	A(G) & A(G) &O& \dots  &A(G)& A(G) \\
	A(G)&O & A(G) & \dots  &A(G)&A(G)\\
	A(G)& A(G) & A(G) & \dots  &A(G)&A(G)\\
	\end{bmatrix}_{pm}\notag\end{align}\begin{align}
	=&\begin{bmatrix}
	0 & 1 & 1& \dots &1& 1\\
	1 & 1 &1 & \dots  &0& 1 \\
	\vdots & \vdots & \vdots & \dots & \vdots & \vdots \\
	1 & 1 &0 & \dots  &1& 1 \\
	1 & 0 &1 & \dots  &1& 1 \\
	1 & 1 & 1 & \dots  &1&1\\
	\end{bmatrix}_{m}\otimes A(G)\notag\label{eq14}\\
	=&B\otimes A(G)\notag
	\textnormal{, where}\quad B=\begin{bmatrix}
	0 & 1 & 1& \dots &1& 1\\
	1 & 1 &1 & \dots  &0& 1 \\
	\vdots & \vdots & \vdots & \dots & \vdots & \vdots \\
	1 & 1 &0 & \dots  &1& 1 \\
	1 & 0 &1 & \dots  &1& 1 \\
	1 & 1 & 1 & \dots  &1&1\\
	\end{bmatrix}_m.\notag
		\end{align}
	Let $X^*=\begin{bmatrix}
	    \frac{(m-3)+\sqrt{m^2-2m+5}}{2(m-1)}\\
		\frac{(m-3)+\sqrt{m^2-2m+5}}{2(m-1)}\\
		\vdots\\
		\frac{(m-3)+\sqrt{m^2-2m+5}}{2(m-1)}\\
	   1
	\end{bmatrix}_{m\times 1}$, then $BX^*=\bigg[\frac{(m-1)+\sqrt{m^2-2m+5}}{2}\bigg]X^*$ 
	and \\$Y^*=\begin{bmatrix}
	\frac{(m-3)-\sqrt{m^2-2m+5}}{2(m-1)}\\
	\frac{(m-3)-\sqrt{m^2-2m+5}}{2(m-1)}\\
	\vdots\\
	\frac{(m-3)-\sqrt{m^2-2m+5}}{2(m-1)}\\
	1
	\end{bmatrix}_{m\times 1}$, then $BY^*=\bigg[\frac{(m-1)-\sqrt{m^2-2m+5}}{2}\bigg]Y^*.$\\ Case 1. $m$ is even.\\\\ Let $E_i=\begin{bmatrix}
	-2\\
	e_i\\
	\end{bmatrix}_{m\times 1},1 \leq i \leq \frac{m-2}{2}$, where $e_i$ is the column vector having $i^{th}$ entry and $(m-1-i)^{th}$ entry one and all other entries  zeros. Then $BE_i=-E_i.$
	Let $E_i^*=\begin{bmatrix}
	0\\
	e_i^*\\
	\end{bmatrix}_{m\times 1},1 \leq i \leq\frac{m-2}{2}$ where $e_i^*$ is the column vector having $i^{th}$ entry  $-1$ , $(m-1-i)^{th}$ entry  $1$ and all other entries  zeros. Then $BE_i^*=E_i^*.$ \\
	\[spec(B)=\begin{pmatrix}
	\frac{(m-1)+\sqrt{m^2-2m+5}}{2} &\frac{(m-1)-\sqrt{m^2-2m+5}}{2}&-1&1\\
	1 &1&\frac{m-2}{2}&\frac{m-2}{2} 
	\end{pmatrix}.\]\\Hence	\[spec(F_1^m(G))=\begin{pmatrix}
	\frac{(m-1)+\sqrt{m^2-2m+5}}{2}\lambda_i &\frac{(m-1)-\sqrt{m^2-2m+5}}{2}\lambda_i&-\lambda_i&\lambda_i\\
	1 &1&\frac{m-2}{2}&\frac{m-2}{2} 
	\end{pmatrix},1\leq i\leq p.\]
	\\ Case 2. $m$ is odd.\\ Let $T=\begin{bmatrix}
	-1\\
	e_{\frac{m-1}{2}}\\
	\end{bmatrix}_{m\times 1},$ where $e_{\frac{m-1}{2}}$ is the column vector having $(\frac{m-1}{2})^{th}$ entry one and all other entries zeros. Then $BT=-T.$\\
	Let $T_i^*=\begin{bmatrix}
	-2\\
	e_i\\
	\end{bmatrix}_{m\times 1},1 \leq i \leq \frac{m-3}{2}$ where $e_i$ as in the case 1. Then $BT_i^*=-T_i^*.$\\
	Let $T_i^{**}=\begin{bmatrix}
	0\\
	e_i^*\\
	\end{bmatrix}_{m\times 1},1 \leq i \leq\frac{m-3}{2} $ where $e_i^*$ as in the case 1. Then $BT_i^{**}=T_i^{**}.$\\
	So the spectrum of $B$ is,\\  
	\[spec(B)=\begin{pmatrix}
	\frac{(m-1)+\sqrt{m^2-2m+5}}{2} &\frac{(m-1)-\sqrt{m^2-2m+5}}{2}&-1&1\\
	1 &1&\frac{m-1}{2}&\frac{m-3}{2} 
	\end{pmatrix}.\]
	Hence \\   \[spec(F_1^m(G))=\begin{pmatrix}
	\frac{(m-1)+\sqrt{m^2-2m+5}}{2}\lambda_i &\frac{(m-1)-\sqrt{m^2-2m+5}}{2}\lambda_i&-\lambda_i&\lambda_i\\
	1 &1&\frac{m-1}{2}&\frac{m-3}{2} 
	\end{pmatrix},1\leq i\leq p.\]
	Thus the energy of $F_1^m(G)$ is,\\
	$\varepsilon(F_1^m(G)) = \bigg[m-2+\sqrt{m^2-2m+5}\bigg]\varepsilon(G) $.
\end{proof}
The next theorem gives a relation between the Randi{\'c} energy of $F_1^m(G)$ and Randi{\'c} energy of $G$.
\begin{thm}\label{h222}
	The Randi{\'c} energy of the graph $F_1^m(G)$ is, $\varepsilon_R(F_1^m(G)) = \varepsilon_R(G)+\frac{(m-1)\varepsilon_R(G)}{m} $.
\end{thm}
\begin{proof}
	The Randi{\'c} matrix of $F_1^m(G)$  is
	\begin{align*} R(F_1^m(G)) &=\left[\begin{smallmatrix}
	((m-1)D)^{-\frac{1}{2}} &   O & \dots  & O & O\\
	O & ((m-1)D)^{-\frac{1}{2}} &  \dots  & O& O \\
	\vdots & \vdots & \vdots & \vdots & \vdots \\
	O & O & \dots  &((m-1)D)^{-\frac{1}{2}}& O\\
	O & O & \dots  & O & (tD)^{-\frac{1}{2}}\\
	\end{smallmatrix}\right]_{pm}\cdot
	\left[\begin{smallmatrix}
	O &A(G) &  A(G) & \dots &A(G)& A(G)\\
	A(G) & A(G) &A(G) & \dots  &O& A(G) \\
	\vdots & \vdots & \vdots & \dots & \vdots & \vdots \\
	A(G)&O & A(G) & \dots  &A(G)&A(G)\\
	A(G)& A(G) & A(G) & \dots  &A(G)&A(G)\\
	\end{smallmatrix}\right]_{pm}\notag\\&\hspace{4cm}
	\left[\begin{smallmatrix}
	((m-1)D)^{-\frac{1}{2}} &O &  O & \dots  & O& O\\
	O & ((m-1)D)^{-\frac{1}{2}} & O & \dots  &O & O\\
	\vdots & \vdots & \vdots & \ddots & \vdots \\
	O & O &O & \dots &((m-1)D)^{-\frac{1}{2}} & O\\
	O & O &O & \dots  & O&(tD)^{-\frac{1}{2}}& O\\
	\end{smallmatrix}\right]_{pm}\notag\\
	\end{align*}
	\begin{align*}
	& =	\begin{bmatrix}
	0 & \frac{1}{m-1} & \frac{1}{m-1} & \dots &\frac{1}{m-1} & \frac{1}{\sqrt{m(m-1)}}\\
	\frac{1}{m-1} & \frac{1}{m-1} & \frac{1}{m-1} & \dots  &0& \frac{1}{\sqrt{m(m-1)}} \\
	\frac{1}{m-1} & \frac{1}{m-1} & \frac{1}{m-1} & \dots  & \frac{1}{m-1}&\frac{1}{\sqrt{m(m-1)}} \\
	\vdots & \vdots & \vdots & \ddots & \vdots&\vdots \\
	\frac{1}{m-1} & 0& \frac{1}{m-1} & \dots  &\frac{1}{m-1}& \frac{1}{\sqrt{m(m-1)}}\\
	\frac{1}{\sqrt{m(m-1)}} & \frac{1}{\sqrt{m(m-1)}}& \frac{1}{\sqrt{m(m-1)}} & \dots  &\frac{1}{\sqrt{m(m-1)}}& \frac{1}{m}\\
	\end{bmatrix}_{m}\otimes D^{-\frac{1}{2}}A(G)D^{-\frac{1}{2}}\notag\\
	&=C\otimes D^{-\frac{1}{2}}A(G)D^{-\frac{1}{2}}\notag\\
		\textnormal{	where}\quad C=&\begin{bmatrix}
		0 & \frac{1}{m-1} & \frac{1}{m-1} & \dots &\frac{1}{m-1} & \frac{1}{\sqrt{m(m-1)}}\\
	\frac{1}{m-1} & \frac{1}{m-1} & \frac{1}{m-1} & \dots  &0& \frac{1}{\sqrt{m(m-1)}} \\
	\frac{1}{m-1} & \frac{1}{m-1} & \frac{1}{m-1} & \dots  & \frac{1}{m-1}&\frac{1}{\sqrt{m(m-1)}} \\
	\vdots & \vdots & \vdots & \ddots & \vdots&\vdots \\
	\frac{1}{m-1} & 0& \frac{1}{m-1} & \dots  &\frac{1}{m-1}& \frac{1}{\sqrt{m(m-1)}}\\
	\frac{1}{\sqrt{m(m-1)}} & \frac{1}{\sqrt{m(m-1)}}& \frac{1}{\sqrt{m(m-1)}} & \dots  &\frac{1}{\sqrt{m(m-1)}}& \frac{1}{m}\\
	\end{bmatrix}_{m}.\\
	\end{align*}

	Let $X^{**}=\begin{bmatrix}
	1\\
	1\\
	\vdots\\
	1\\
	\sqrt{\frac{m}{m-1}}
	\end{bmatrix}_{m\times 1},$
	then $CX^{**}=X^{**}$ and $Y^{**}=\begin{bmatrix}
	1\\
	1\\
	\vdots\\
	1\\
	-(m-1)\sqrt{\frac{m-1}{m}}\\
	\end{bmatrix}_{m\times 1},$ then  $CY^{**}=-\frac{1}{m(m-1)}Y^{**}$.\\ Case 1. $m$ is even.\\  Let $E_i$ and $E_i^*$  as in Theorem 4.1, then $CE_i=-\frac{1}{m-1}E_i.$
	 and  $CE_i^*=\frac{1}{m-1}E_i^*.$\\ Hence \[RS(F_1^m(G)))=\begin{pmatrix}
	\rho_i &-\frac{\rho_i}{m(m-1)}&\frac{-\rho_i}{m-1}&\frac{\rho_i}{m-1}\\
	1 &1&\frac{m-2}{2}&\frac{m-2}{2} 
	\end{pmatrix},1\leq i\leq p.\]\\ Case 2. $m$ is odd.\\ Let $T$, $T_i^*$ and $T_i^{**}$, as in Theorem 4.1, then $CT=-\frac{1}{m-1}T$, $CT_i^*=-\frac{1}{m-1}T_i^* $ and $CT_i^{**}=\frac{1}{m-1}T_i^{**}.$
	So the simple eigenvalues of $C$ are $1,-\frac{1}{m(m-1)}$ and  $ -\frac{1}{m-1}$ with multiplicity $\frac{m-1}{2}$ times and $\frac{1}{m-1}$ with multiplicity $\frac{m-3}{2}$ times.
	
	From this we can see that the Randi{\'c} spectrum of $F_1^m(G)$ is\\
	\[RS(F_1^m(G)))=\begin{pmatrix}
	\rho_i &-\frac{\rho_i}{m(m-1)}&-\frac{\rho_i}{m-1}&\frac{\rho_i}{m-1}\\
	1 &1&\frac{m-1}{2}&\frac{m-3}{2} 
	\end{pmatrix},1\leq i\leq p.\]
	Hence $\varepsilon_R(F_1^m(G)) = \varepsilon_R(G)+\frac{(m-1)\varepsilon_R(G)}{m} $.
\end{proof}

\begin{op}\label{dh3}
	Let $G$ be a simple $(p,q)$ graph with vertex set $V(G)$ and edge set $E(G)$. Let $D_m(G),m\geq3$ be the $m$-shadow graph of $G$ and $G_1, G_2,...,G_m$ be the m copies of $G$ in $D_m(G)$. Define $F_2^m(G)=D_m(G)-E(G_i), 2\leq i\leq m.$ 
\end{op}
The number of vertices and edges in  $F_2^m(G)$ are $pm$ and  $(m^2-m+1)q$ respectively.

\begin{thm}\label{h31}
	The energy of the graph $F_2^m(G)$ is, $\varepsilon(F_2^m(G)) = \bigg[m-2+\sqrt{m^2-2m+5}\bigg]\varepsilon(G) $.
\end{thm}
\begin{proof}
	With the suitable labeling of the vertices, the adjacency matrix of $F_2^m(G)$ is 
	\begin{align}
	A(F_2^m(G)) =&\begin{bmatrix}
	A(G) &A(G) &   \dots & A(G)\\
	A(G) & O & \dots  & A(G) \\
	\vdots & \vdots &  \dots &  \vdots \\
	A(G)& A(G) &  \dots  &O\\
	\end{bmatrix}_{pm}\notag\\
	=&\begin{bmatrix}
	1 & 1 & 1& \dots &1& 1\\
	1 & 0 &1 & \dots  &1& 1 \\
	\vdots & \vdots & \vdots & \dots & \vdots & \vdots \\
	1 & 1 & 1 & \dots  &1&0\\
	\end{bmatrix}_{m}\otimes A(G)\label{eq14}\notag\\
	=&H\otimes A(G),
	\textnormal{	where } H=\begin{bmatrix}
	1 & 1 & 1& \dots &1& 1\\
	1 & 0 &1 & \dots  &1& 1 \\
	\vdots & \vdots & \vdots & \dots & \vdots & \vdots \\
	1 & 1 & 1 & \dots  &1&0\\
	\end{bmatrix}_m.\notag
	\end{align}
	Let $P=\begin{bmatrix}
	\frac{(3-m)+\sqrt{m^2-2m+5}}{2}\\
	1\\
	1\\
	\vdots\\
	1
	\end{bmatrix}_{m\times 1}$
	and $Q=\begin{bmatrix}
	\frac{(3-m)-\sqrt{m^2-2m+5}}{2}\\
	1\\
	1\\
	\vdots\\
	1
	\end{bmatrix}_{m\times 1},$
	then $HP=\bigg[\frac{(m-1)+\sqrt{m^2-2m+5}}{2}\bigg]P$  and $HQ=\bigg[\frac{(m-1)-\sqrt{m^2-2m+5}}{2}\bigg]Q.$ \\Let $F_i=\begin{bmatrix}
	0\\
	-1\\
	f_i\\
	\end{bmatrix}_{m\times 1},1 \leq i \leq m-2$, where $f_i$ is the column vector having $i^{th}$ entry one and all other entries zeros. Then $HF_i=-F_i.$
	So the simple eigenvalues of $H$ are  $\frac{(m-1)+\sqrt{m^2-2m+5}}{2},\frac{(m-1)-\sqrt{m^2-2m+5}}{2}$, and $-1$ with multiplicity $m-2$ times. Hence
	\[spec(F_2^m(G))=\begin{pmatrix}
	\frac{(m-1)+\sqrt{m^2-2m+5}}{2}\lambda_i &\frac{(m-1)-\sqrt{m^2-2m+5}}{2}\lambda_i&-\lambda_i\\
	1 &1&m-2 
	\end{pmatrix},1\leq i\leq p.\]
	Thus we have  $\varepsilon(F_2^m(G)) = \bigg[m-2+\sqrt{m^2-2m+5}\bigg]\varepsilon(G) $.
\end{proof}
\begin{thm}\label{h32}
	The Randi{\'c} energy of the graph $F_2^m(G)$  is, $\varepsilon_R(F_2^m(G)) = \varepsilon_R(G)+\frac{(m-1)\varepsilon_R(G)}{m} $.
\end{thm}
\begin{proof}
	The Randi{\'c} matrix of  $F_2^m(G)$  is
	\begin{align} R(F_2^m(G)) &=\left[\begin{smallmatrix}
	(tD)^{-\frac{1}{2}} & O &  \dots  & O\\
	O & ((m-1)D)^{-\frac{1}{2}} & \dots  & O \\
	\vdots & \vdots &  \ddots & \vdots \\
	O & O &  \dots  & ((m-1)D)^{-\frac{1}{2}}\\
	\end{smallmatrix}\right]_{pm}\cdot
	\left[\begin{smallmatrix}
	A(G) & A(G) &   \dots  & A(G)\\
	A(G) & O & \dots  & A(G) \\
	\vdots & \vdots &  \ddots & \vdots \\
	A(G) & A(G) & \dots  & O\\
	\end{smallmatrix}\right]_{pm}\notag\\&\hspace{4cm}
	\left[\begin{smallmatrix}
	(tD)^{-\frac{1}{2}} &O &  O & \dots  & O\\
	O & ((m-1)D)^{-\frac{1}{2}} & O & \dots  & O \\
	\vdots & \vdots & \vdots & \ddots & \vdots \\
	O & O & O & \dots  & ((m-1)D)^{-\frac{1}{2}}\\
	\end{smallmatrix}\right]_{pm}\notag\\
	& =	\begin{bmatrix}
	\frac{1}{m} & \frac{1}{\sqrt{m(m-1)}} & \frac{1}{\sqrt{m(m-1)}} & \dots  & \frac{1}{\sqrt{m(m-1)}}\\
	\frac{1}{\sqrt{m(m-1)}} & 0 & \frac{1}{m-1} & \dots  & \frac{1}{m-1} \\
	\vdots & \vdots & \vdots & \ddots & \vdots \\
	\frac{1}{\sqrt{m(m-1)}} & \frac{1}{m-1} & \frac{1}{m-1} & \dots  & 0\\
	\end{bmatrix}_{m}\otimes D^{-\frac{1}{2}}A(G)D^{-\frac{1}{2}}\label{eq5}\\
	=&L\otimes D^{-\frac{1}{2}}A(G)D^{-\frac{1}{2}},\notag
	\text{ where }L  =\begin{bmatrix}
	\frac{1}{m} & \frac{1}{\sqrt{(m-1)m}} & \frac{1}{\sqrt{(m-1)m}} & \dots  & \frac{1}{\sqrt{(m-1)m}}\\
	\frac{1}{\sqrt{(m-1)m}} & 0 & \frac{1}{m-1} & \dots  & \frac{1}{m-1} \\
	\vdots & \vdots & \vdots & \ddots & \vdots \\
	\frac{1}{\sqrt{(m-1)m}} & \frac{1}{m-1} & \frac{1}{m-1} & \dots  & 0\\
	\end{bmatrix}_{m}.\notag
	\end{align}	
	Let $P^*=\begin{bmatrix}
	\sqrt{\frac{m}{m-1}}\\
	1\\
	1\\
	\vdots\\
	1
	\end{bmatrix}_{m\times 1}$
	 and  $Q^*=\begin{bmatrix}
	-(m-1)\sqrt{\frac{m-1}{m}}\\
	1\\
	1\\
	\vdots\\
	1
	\end{bmatrix}_{m\times 1},$
	then $LP^*=1.P^*$ and $LQ^*=-\frac{1}{m(m-1)}Q^*$. Let $F_i$ as in Theorem 4.3. Then $LF_i=-\frac{1}{m-1}F_i.$
	So the simple eigenvalues of $L$ are $1,-\frac{1}{m(m-1)}$ and $ -\frac{1}{m-1}$ with multiplicity $m-2$ times. Thus
	\[RS(F_2^m(G))=\begin{pmatrix}
	\rho_i &-\frac{1}{m(m-1)}\rho_i&-\frac{1}{m-1}\rho_i\\
	1 &1&m-2 
	\end{pmatrix},1\leq i\leq p.\]
	Hence we get $\varepsilon_R(F_2^m(G)) = \varepsilon_R(G)+\frac{(m-1)\varepsilon_R(G)}{m} $.
\end{proof}
From the following Propositions  we can construct a family of graphs which  are both equienergetic and Randi{\'c} equienergetic non-cospectral graphs. 
\begin{prop}
	Let $G$ be a simple $(p,q)$  graph. Then the graphs $F_1^m(G)$ and $F_2^m(G)$  are equienergetic non-cospectral graphs for every  $m$.
\end{prop}
\begin{proof}
	Proof follows from Theorems \ref{h111} and \ref{h31} .
\end{proof}
\begin{prop}
	Let $G$ be a  simple $(p,q)$  graph. Then the graphs $F_1^m(G)$ and $F_2^m(G)$  are Randi{\'c} equienergetic non-cospectral graphs for every $m$.	
\end{prop}
\begin{proof}
	Proof follows from Theorems \ref{h222} and \ref{h32} . 
\end{proof}	
The following Proposition gives an infinite family of Randi{\'c} equienergetic graphs from a given pair of Randi{\'c} equienergetic regular graphs.  
\begin{prop}
	Let $G_1$ and $G_2$ be two $r_1-$regular Randi{\'c} equienergetic graphs non-cospectral graphs, then the $k^{th}$ iterated line graphs $L^k(G_1)$ and  $L^k(G_2)$ are Randi{\'c} equienergetic non-cospectral graphs.
	\begin{proof}
		Proof follows from Theorem 2.1 and Proposition 2.3.
	\end{proof}
\end{prop}
Proposition 4.4 gives an infinite family of Randi{\'c} equienergetic graphs from a given pair of  Randi{\'c} equienergetic graphs.
\begin{prop}
	Let $G_1$ and $G_2$ be two Randi{\'c} equienergetic non-cospectral graphs, and $H$ be  any graph, then the graphs $H\times G_1$ and $H\times G_2$ are Randi{\'c} equienergetic non-cospectral graphs.
	\begin{proof}
		If $G_1$ and $G_2$ are Randi{\'c} equienergetic, so are $H\times G_1$ and $H\times G_2$ for any graph $H$, since $\varepsilon_R(H\times G_1)=\varepsilon_R(H).\varepsilon_R (G_1)=\varepsilon_R(H).\varepsilon_R (G_2)=\varepsilon_R(H\times G_2)$\cite{butler2016algebraic} and the resulting graphs have the same order. Also if $G_1$ and $G_2$ are non-cospectral, then $H\times G_1$ and $H\times G_2$ are non-cospectral.
	\end{proof}
\end{prop}

\section{Application}
 In this section, we construct some sequence of graphs satisfying the property (R) (respectively, (SR), (-R), (-SR)). Also, we obtain some class of graphs with property (R) but not (SR).\\

The following theorem helps us to construct graphs which satisfies the property (-R) but not (-SR).
\begin{thm}
	Let $G$  be a graph  satisfying the property (SR) and $m$ be an odd positive integer. Then  $F_1^m(G)$ satisfies the property (-R) but not  (-SR).
\end{thm}
\begin{proof}Let $\lambda\in  spec(G)$ and  $m$ be an odd positive integer. Then                     \[spec(F_1^m(G))=\begin{pmatrix}
	\frac{(m-1)+\sqrt{m^2-2m+5}}{2}\lambda &\frac{(m-1)-\sqrt{m^2-2m+5}}{2}\lambda&-\lambda&\lambda\\
	1 &1&\frac{m-1}{2}&\frac{m-3}{2} 
	\end{pmatrix}.\]
	Since $G$  satisfies the property (SR), corresponding to each eigenvalue $\lambda$ of $G$, $\frac{1}{\lambda}$ is also an eigenvalue of $G$ with the same multiplicity. Let $\alpha= \frac{(m-1)+\sqrt{m^2-2m+5}}{2}\lambda$, then $-\dfrac{1}{\alpha}=\frac{-2}{(m-1)+\sqrt{m^2-2m+5}}\frac{1}{\lambda}=\frac{(m-1)-\sqrt{m^2-2m+5}}{2}(\frac{1}{\lambda})\in spec(F_1^m(G))$ as $G$  satisfies the property (SR).  \\Similarly for $\beta= \frac{(m-1)-\sqrt{m^2-2m+5}}{2}\lambda$, $-\dfrac{1}{\beta}=\frac{-2}{(m-1)-\sqrt{m^2-2m+5}}\frac{1}{\lambda}=\frac{(m-1)+\sqrt{m^2-2m+5}}{2}(\frac{1}{\lambda})\in spec(F_1^m(G)).$ Also $\frac{1}{\lambda}\in spec(F_1^m(G))$ and $-\frac{1}{\lambda}\in spec(F_1^m(G))$ as $G$ satisfies the property (SR).
	Thus $F_1^m(G)$ satisfies the property (-R).
   	 Since $\lambda$ and $-\lambda$ has different multiplicity in $F_1^m(G)$,
	we have $\lambda$ and $-\frac{1}{\lambda}$ has different multiplicity in $F_1^m(G)$. Thus $F_1^m(G)$ satisfies the property (-R) but not  (-SR).
\end{proof}

The eigenvalues given in each examples are decimal approximations calculated by the help of mathlab.
\begin{exam}
 Let $G$ be a graph in Figure 6. 	
\end{exam}	
\begin{figure}[h!]
	\centering
	\includegraphics[width=8.0 cm]{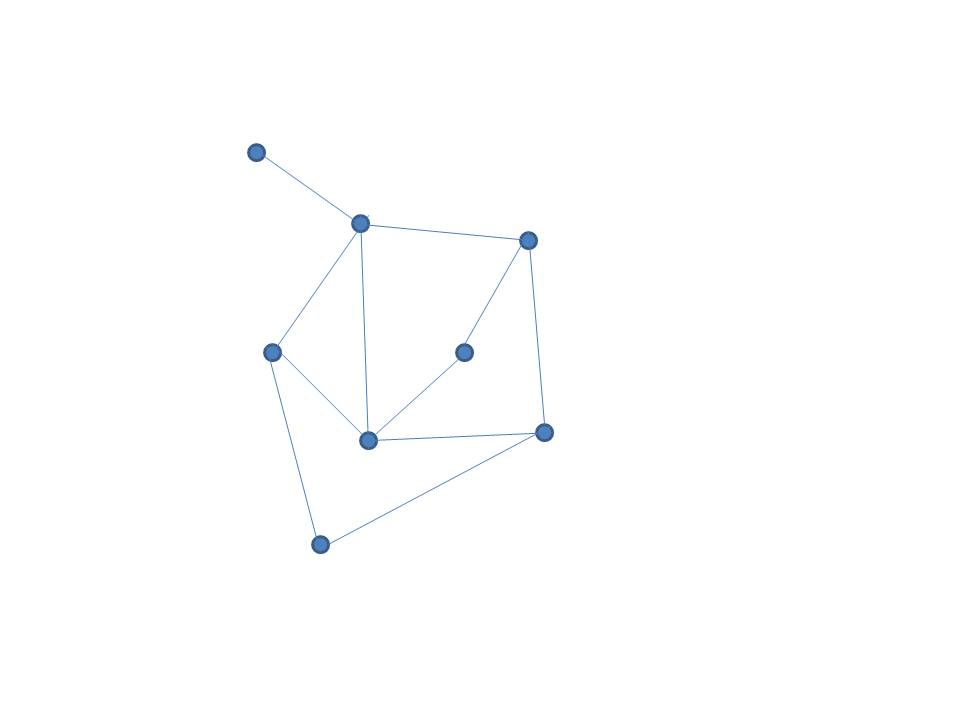}
	\caption{  $G$}
	\label{pict8.jpg}
	
\end{figure} Then	
\[spec(G)=\begin{pmatrix}
1&\frac{-3\pm\sqrt{5}}{2}&\frac{1+\sqrt{33}\pm \sqrt{18+2\sqrt{33}}}{4}& 	\frac{1-\sqrt{33}\pm \sqrt{18-2\sqrt{33}}}{4}\\
2&1&1&1\\
\end{pmatrix}.\]spec($F_1^5(G))$
\[=\begin{pmatrix}
1&-1&12.8935&-0.0776& -11.0902&0.0902&-7.7268&0.1294\\2&4&1&1&1&1&1&1\\\\
4.2361&-0.2361&-2.3223&0.4306&-3.0437&0.3285&3.0437&-0.3285\\
2&2&1&1&2&2&1&1\\\\-1.6180&0.6180&-2.6180&0.3820&2.6180&-0.3820&1.3917&-0.7185\\\\
1&1&1&2&2&1&1&1\\\\-1.8241&0.5482&-0.5482&1.8241\\1&2&1&2
\end{pmatrix}.\]

Note that $G$ in Figure 6 satisfy the property (SR) . The graph  $F_1^5(G)$ satisfies the property (-R) but not (-SR).\\
\begin{thm}
	Let $G$ be a graph satisfying the property (-SR) and $m$ be an odd positive integer. Then the graph $F_1^m(G)$  satisfies the property (R) but not (SR).
\end{thm}
\begin{proof}
	Proof is similar to Theorem 5.1.
\end{proof}
The following theorem helps us to construct a new family of graphs which satisfies the  property (-SR).
\begin{thm}
	Let $m$ be an even positive integer. Then $G$  has the property (SR)	if and only if $F_1^m(G)$ has the property (-SR).
\end{thm}
\begin{proof}
	Let $\lambda\in  spec(G)$ and $m$ be even positive integer. Then
	\begin{align*}
	 spec(F_1^m(G))=&\begin{pmatrix}
	\frac{(m-1)+\sqrt{m^2-2m+5}}{2}\lambda &\frac{(m-1)-\sqrt{m^2-2m+5}}{2}\lambda&-\lambda&\lambda\\
	1 &1&\frac{m-2}{2}&\frac{m-2}{2} 
	\end{pmatrix}.\\
	\end{align*}  Since $G$  satisfies the property (SR), corresponding to each eigenvalue $\lambda$ of $G$, $\frac{1}{\lambda}$ is also an eigenvalue of $G$ with the same multiplicity. Let $\alpha= \frac{(m-1)+\sqrt{m^2-2m+5}}{2}\lambda$, then $-\dfrac{1}{\alpha}=\frac{-2}{(m-1)+\sqrt{m^2-2m+5}}\frac{1}{\lambda}=\frac{(m-1)-\sqrt{m^2-2m+5}}{2}(\frac{1}{\lambda})\in spec(F_1^m(G))$ as $G$  satisfies the property (SR).  Similarly for $\beta= \frac{(m-1)-\sqrt{m^2-2m+5}}{2}\lambda$, $-\dfrac{1}{\beta}=\frac{-2}{(m-1)-\sqrt{m^2-2m+5}}\frac{1}{\lambda}=\frac{(m-1)+\sqrt{m^2-2m+5}}{2}(\frac{1}{\lambda})\in spec(F_1^m(G)).$ Also $\frac{1}{\lambda}\in spec(F_1^m(G))$ and $-\frac{1}{\lambda}\in spec(F_1^m(G))$ as $G$ satisfies the property (SR).\\
	Thus $F_1^m(G)$ satisfies the property (-R).  Note that multiplicity of each eigenvalue and its negative reciprocal in $F_1^m(G)$ are same.
	Thus $F_1^m(G)$ satisfies the property (-SR).\\
	Converse also follows by similar arguments.
\end{proof}
\begin{exam}
	
\end{exam}
Let $G$ be a graph in Figure 6. Then\\ spec($F_1^4(G))$
\[=\begin{pmatrix}
1&-1&10.0528&-0.0995& -8.6468&0.1156&-6.0244&0.166\\2&2&1&1&1&1&1&1\\\\
3.0437&-0.3285&-3.0437&0.3285&2.6180&-0.3820&1.8241&-0.5482\\
1&1&1&1&1&1&1&1\\\\-1.8241&0.5482&3.3028&-0.3028&2.6180&-0.3820&-1.2615&0.7927\\\\
1&1&2&2&1&1&1&1\\\\1.0851&-0.92157&-1.8241&0.5482\\1&1&1&1
\end{pmatrix}.\]
Note that $F_1^4(G)$ satisfies property (-SR).
\begin{thm}
	Let $m$ be an even positive integer. Then $G$  has the property (-SR)	if and only if $F_1^m(G)$ has the property (SR).
\end{thm}
\begin{proof}
	Proof is similar to Theorem 5.3.
\end{proof}

\begin{exam}
\end{exam}
\begin{figure}[h!]
	\begin{minipage}[b]{0.5\linewidth}
		\centering
		\includegraphics[width=5.0 cm]{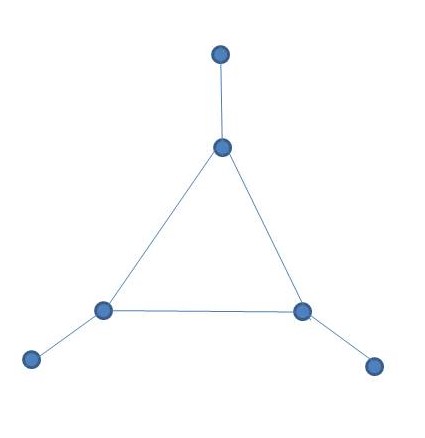}
		\caption{  $K_3\circ K_1$}
		\label{pict12.jpg}
	\end{minipage}
	\hspace{0.5cm}
	\begin{minipage}[b]{0.4\linewidth}
		\centering
		\includegraphics[width=8.0 cm]{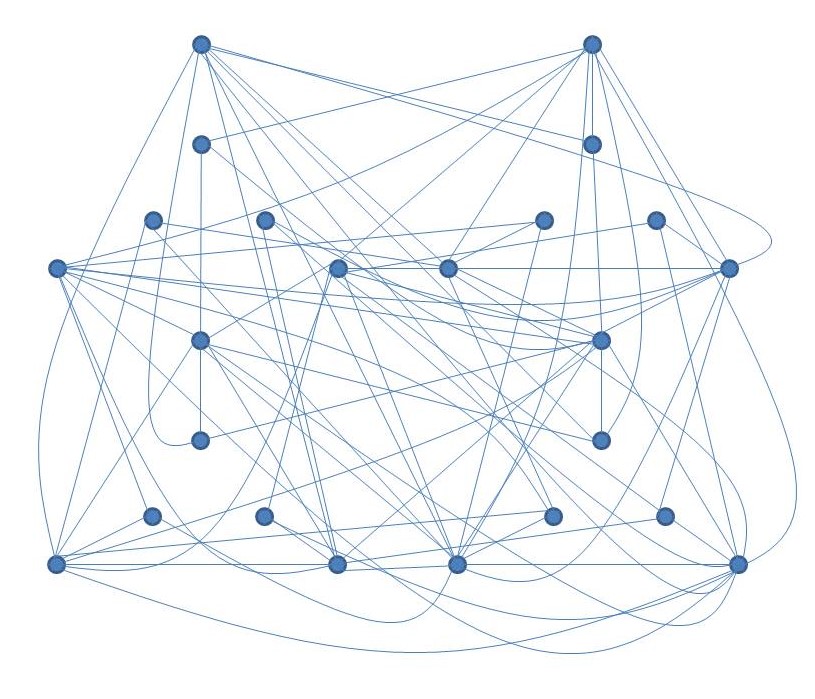}
		\caption{  $F_1^4(K_3\circ K_1)$}
		\label{pict13.jpg}
	\end{minipage}
\end{figure}
\[spec(K_3\circ K_1)=\begin{pmatrix}
2.4142 &-0.4142&-1.6180&0.6180\\
1 &1&2&2 
\end{pmatrix}.\]
\[spec(F_1^4(K_3\circ K_1))=\begin{pmatrix}
2.4142 &0.4142&-5.3440&-0.1871&-2.4142&-0.4142\\
1 &1&2&2&1&1 \\\\-1.3681&-0.7310 &0.1254&7.9736&2.0412&0.4899\\
1 &1&1&1&2&2 \\\\1.6180&0.6180 &-1.6180&-0.6180\\
2 &2&2&2 \\\\
\end{pmatrix}.\]

Note that $K_3\circ K_1$ satisfy the property (-SR), but the graph  $F_1^4(K_3\circ K_1)$  satisfies the property (SR).\\\\
The following theorem helps us to construct a sequence of graphs satisfies the property (R) .

\begin{thm}
	Let $G$ be a  graph satisfying the property (R). Then $F_2^m(G)$ satisfies the property (R) if and only if $G$ is bipartite.
\end{thm}
\begin{proof}
	Let $\lambda\in  spec(G)$. Then \[spec(F_2^m(G))=\begin{pmatrix}
	\frac{(m-1)+\sqrt{m^2-2m+5}}{2}\lambda &\frac{(m-1)-\sqrt{m^2-2m+5}}{2}\lambda&-\lambda\\
	1 &1&m-2 
	\end{pmatrix}.\]
	Since $G$  is bipartite and  satisfies the property (R), corresponding to each eigenvalue $\lambda$ of $G$ $-\lambda$,$\frac{1}{\lambda}$,$-\frac{1}{\lambda}$ are also an eigenvalue of $G$. 
	Let $\alpha= \frac{(m-1)+\sqrt{m^2-2m+5}}{2}\lambda$, then $\dfrac{1}{\alpha}=\frac{2}{(m-1)+\sqrt{m^2-2m+5}}\frac{1}{\lambda}=\frac{(m-1)-\sqrt{m^2-2m+5}}{2}(\frac{-1}{\lambda})\in spec(F_2^m(G))$ as $G$ is bipartite and  satisfies the property (R).  Similarly for $\beta= \frac{(m-1)-\sqrt{m^2-2m+5}}{2}\lambda$, $\dfrac{1}{\beta}=\frac{2}{(m-1)-\sqrt{m^2-2m+5}}\frac{1}{\lambda}=\frac{(m-1)+\sqrt{m^2-2m+5}}{2}(-\frac{1}{\lambda})\in spec(F_2^m(G)).$
	 Also  $-\frac{1}{\lambda}\in spec(F_2^m(G))$.
	Thus $F_2^m(G)$ satisfies the property (R).\\ 
	Conversely assume that $F_2^m(G)$ satisfies the property (R). Then $\frac{(m-1)+\sqrt{m^2-2m+5}}{2}\lambda\in spec(F_2^m(G))$ implies that  $\frac{(m-1)-\sqrt{m^2-2m+5}}{2}(-\frac{1}{\lambda})\in spec(F_2^m(G))$. From $spec(F_2^m(G))$, we get $\frac{-1}{\lambda}\in spec(G)$. By  property (R) of $G$, $-\lambda \in spec(G)$. Thus $\lambda$ and $-\lambda$ are eigenvalues of $G$. Therefore, $G$ is bipartite.    
\end{proof}
\begin{exam}
	
\end{exam}
\[spec(P_4)=\begin{pmatrix}
-1.6180&-0.6180&1.6180&0.6180\\
1 &1&1&1 
\end{pmatrix}\]
\[spec(F_2^3(P_4))=\begin{pmatrix}
-3.9063 &-0.2560&3.9063&0.2560&1.4921&0.6702\\\\1 &1&1&1&1&1\\\\-1.4921&-0.6702&1.6180&0.6180&-1.6180&-0.6180\\1&1 &1&1&1&1
\end{pmatrix}.\]Note that, $P_4$ is bipartite and satisfying property(R) and $F_2^3(P_4)$ satisfying the property (R) .\\

The following example illustrates that if $G$ is non-bipartite, then $F_2^m(G)$ need not satisfies the property (R).
\begin{exam}Let $H$ be a graph in Figure 9. 
	\begin{figure}[H]
			\centering
			\includegraphics[width=10.0 cm]{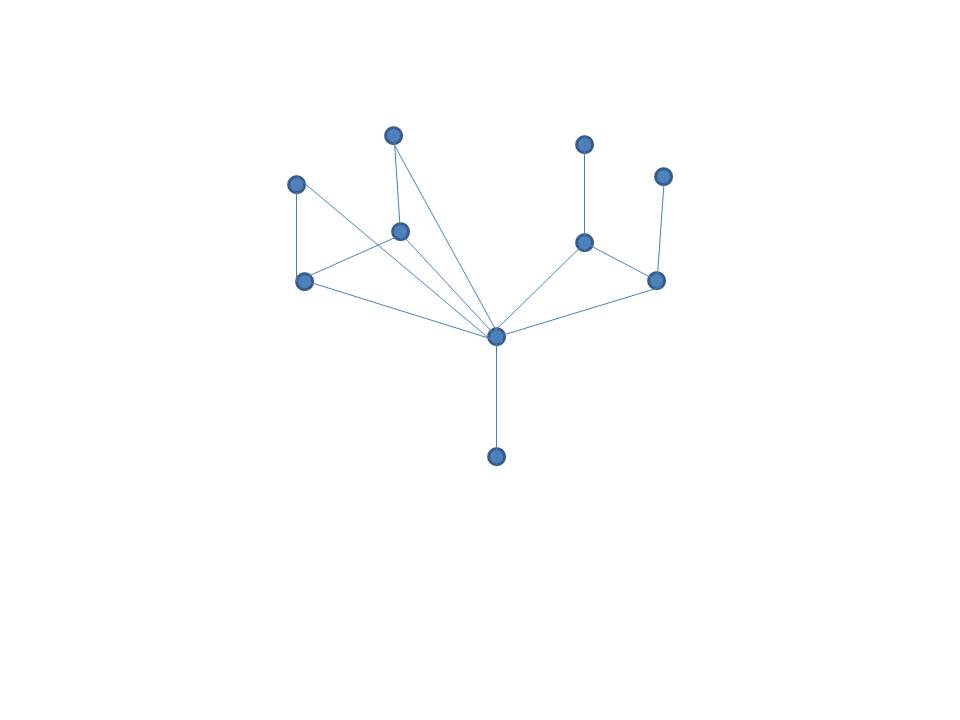}
			\caption{  H}
			\label{pict13.jpg}
		\end{figure}	
		Then\[spec(H)=\begin{pmatrix}
		3.4081 &0.2934&-2.2589&-0.4427&1.6180&0.6180&-1.6180&-0.6180\\\\1 &1&1&1&1&2&2&1
		\end{pmatrix}\]
		\[spec(F_2^3(H))=\begin{pmatrix}
		8.2280 &-0.1215&-5.4534&0.1834&3.9063&0.2560&-3.9063&-0.2560\\
		1 &1&1&1&1&1&2&2 \\\\-3.4081&-0.2934 &-1.4117&0.7084&-1.0688&0.9357&2.2589&0.4427\\
		1 &1&1&1&1&1&1&1 \\\\-1.6180&0.6180 &1.6180&-0.6180&-1.4921&-0.6702&1.4921&0.6702\\
		1 &1&2&2&1&1&2&2 \\\\
		\end{pmatrix}.\]
		
\end{exam}
Note that $H$  is non-bipartite and satisfies the property (R) but $F_2^3(H)$ doesn't  satisfy  the property (R). \\\\\par The following theorem helps us to construct a sequence of graphs satisfies the property (-SR).
\begin{thm}
	Let $G$ be a  graph satisfying the property (-SR). Then $F_2^m(G)$ satisfies the property (-SR) if and only if $G$ is bipartite.
\end{thm}
\begin{proof}
	By the same argument as in Theorem 5.5.
\end{proof}

\section{Conclusion}
The concept of Randi{\'c} equienergetic graphs is analogous to the concept of equienergetic graphs . We provide some new methods for constructing families of  equienergetic and Randi{\'c} equienergetic graphs based on $S(G)_k$ and $S(G)_t^n$. Based on these graphs, we also construct some new families of integral graphs. In addition, some new families of equienergetic and Randi{\'c} equienergetic graphs are obtained by using the graphs $F_1^m(G)$ and $F_2^m(G)$. Moreover, a sequence of graphs with reciprocal eigenvalue property and anti-reciprocal eigenvalue property are established.

\bibliography{refe1}
\bibliographystyle{plain}

\end{document}